\documentclass{article}
\usepackage[utf8]{inputenc}
\usepackage{graphicx}
\usepackage{color}
\usepackage{multicol} 
\usepackage{amsmath}
\usepackage{amsthm}
\usepackage{amssymb}
\usepackage{fullpage}
\usepackage{cases}
\usepackage{stmaryrd}
\usepackage{marvosym} 
\usepackage{manfnt} 
\definecolor{mygreen}{RGB}{46,139,87}
\usepackage{subfig}
\usepackage{setspace}

\usepackage{todonotes} 

\newcommand{\AddresseLAGA}{{
  \bigskip
  \footnote{
  LAGA, Universit\'e Paris 13, Villetaneuse, France}}}
  \newcommand{\AddressePOEMS}{{
    \bigskip
    \footnote{
    POEMS (CNRS-ENSTA Paristech-INRIA, Universit\'e Paris-Saclay), 828 Boulevard des Mar\'echaux, Palaiseau, France}}}
\newcommand{\Corr}{\footnote{
{Corresponding author. }
  \textit{E-mail address}: \texttt{sonia.fliss@ensta-paristech.fr}
}
}
\newtheorem{prop}{Proposition}
\newtheorem{thm}{Theorem}
\newtheorem{lemme}{Lemma}
\newtheorem{rem}{Remark}
\newtheorem{corollaire}{Corollary}
\def\N{\mathbb{N}}
\def\Z{\mathbb{Z}}
\def\R{\mathbb{R}}

\def\eps{\varepsilon}

\def\dsp{\displaystyle}
\begin{document}
	\title{Trapped modes in thin and infinite ladder like domains. Part 1 : existence results}
\author{B\'erang\`ere Delourme\AddresseLAGA, Sonia Fliss\AddressePOEMS\Corr, Patrick Joly\AddressePOEMS, Elizaveta Vasilevskaya\AddresseLAGA}
	\maketitle
\abstract{
\noindent The present paper deals with the wave propagation in a particular two dimensional structure, obtained from a localized perturbation of a reference periodic medium. This reference medium is a ladder like domain, namely a thin periodic structure (the thickness being characterized by a small parameter $\eps > 0$) whose limit (as $\eps$ tends to 0) is a periodic graph. The localized perturbation consists in changing the geometry of the reference medium by modifying the thickness of one rung of the ladder. Considering the scalar Helmholtz equation with Neumann boundary conditions in this domain,  we wonder whether such a geometrical perturbation is able to produce localized eigenmodes.  To address this question, we use a standard approach of asymptotic analysis that consists of three main steps. We first find the formal limit of the eigenvalue problem as the $\varepsilon$ tends to 0. In the present case, it corresponds to an eigenvalue problem for a second order differential operator defined along the periodic graph. Then, we proceed to an explicit calculation of the spectrum of the limit operator. Finally, we prove that the spectrum of the initial operator is close to the spectrum of the limit operator.  In particular, we prove the existence of localized modes provided that the geometrical perturbation consists in diminishing the width of one rung of the periodic thin structure. Moreover, in that case, it is possible to create as many eigenvalues as one wants, provided that $\varepsilon$ is small enough. Numerical experiments illustrate the theoretical results.

~\\\\\textit{Keywords:
	spectral theory, periodic media, quantum
  graphs, trapped modes
}
\section{Introduction}
Photonic crystals, also known as electromagnetic bandgap metamaterials,
are 2D or 3D periodic media designed to control the light
propagation. Indeed, the multiple scattering resulting from the periodicity of
the material can give rise to destructive interferences at some range
of frequencies.  It follows that there
might exist intervals of frequencies (called gaps) wherein the monochromatic waves
cannot propagate. At the same time, a local perturbation of the
crystal can produce  defect mid-gap modes, that is to say
solutions to the homogeneous time-harmonic wave equation, at a fixed frequency located inside one gap, 
that remains strongly localized in the vicinity of the
perturbation. This localization phenomenon is of particular interest
for a variety of promising
applications in optics, for instance the design of highly efficient waveguides \cite{Joannopoulos:1995,Johnson:2002}. \\
~\\\\
\noindent From a mathematical point of view, the presence of gaps is theoretically
explained by the band-gap structure of the spectrum of the periodic partial
differential operator associated with the wave propagation in such
materials (see for instance \cite{Eastham:1973,
  Kuchment:1993}).  In turn, the localization effect is directly linked to the
possible presence of discrete spectrum appearing when perturbing the
perfectly periodic operator. A thorough mathematical description of
photonic crystals can be found in \cite{Kuchment:2001}. Without being exhaustive, let us
remind the reader about a few important results on the topic.
In the one dimensional case, it is
well-known \cite{Borg:1946} that a periodic
material has infinitely many gaps unless it is constant. By contrast,
in 2D and 3D, a periodic medium might or might not have
gaps. Nevertheless, several configurations where at least one gap
do exist can be
found in \cite{Figotin:1996a,Figotin:1996b,HoangPlumWieners2009,Nazarov:2010a,Nazarov:2012a,NazarovBakharevRuotsalainen:2013,Khrabustovskyi2014,KhrabustovskyiKhruslov2015} and
references therein. In
any case, except in dimension one, the number of gaps is expected to
be finite. This statement, known as the Bethe Sommerfeld conjecture is
fully demonstrated in 
\cite{Parnovski:2008,Parnovski:2010} for the periodic Schrödinger operator
but is still partially open for Maxwell equations (see
\cite{Vorobets:2011}). For the localization effect, \cite{Figotin:1997,Figotin:1998b,Ammari:2004,Kuchment:2003,KuchmentOng:2010,NazarovLadder2014, BrownHoangPlumWood2014,BrownHoangPlumWood2015}
several papers exhibit situations where a compact (resp. lineic)  perturbation of a
periodic medium give rise to localized (resp. guided) modes. It seems that the first results concern strong material perturbations : for local perturbation \cite{Figotin:1997,Figotin:1998b} and for lineic perturbation \cite{Kuchment:2003,KuchmentOng:2010}. There exist fewer results about weak material perturbations : \cite{BrownHoangPlumWood2014,BrownHoangPlumWood2015} deal with 2D lineic perturbations. Finally, geometrical perturbations are considered {{in}}~\cite{NazarovLadder2012,NazarovLadder2014}, { where the geometrical domain under investigation is exactly the same as ours but with homogeneous Dirichlet boundary conditions on the boundary of the ladder.} 
As in our case,  changing the size of one or several rungs of the ladder can create eigenvalues inside a gap {(see also remark~\ref{remarkDirichletLadder})}. 
~\\\\
\noindent The aim of this paper is to complement the references
mentioned above by proving the existence of localized midgap
modes created by a geometrical perturbation of a particular periodic medium. We
shall use a standard approach of analysis (used in
\cite{Figotin:1996a,Nazarov:2012a}) that consists in
comparing the periodic medium
with a reference one, for which
theoretical results are available. To be more specific, we are interested in the
Laplace operator with Neumann boundary condition
in a ladder-like periodic
waveguide.  As the thickness of the rungs (proportional to a small parameter $\varepsilon$)  tends to zeros, the domain
shrinks to an (infinite) periodic graph. More precisely, the spectrum of the
operator posed on the 2D domain tends to the spectrum of a
self-adjoint operator posed on the limit graph
(\cite{SchatzmannRubinstein:2001,KuchmentZeng:2001,Saito:2000,Post:2006,panasenko2007}). 
This limit operator consists of the second order derivative operator on
each edge of the graph together with transmission conditions (called
Kirchhoff conditions) at
its vertices (\cite{Exner:1996,carlson:1998,KuchmentZeng:2001}). As
opposed to the initial operator, the spectrum of the limit operator
can be explicitly determined
using a finite difference scheme (\cite{AvishaiLuck:1992,Exner:1995}). From a mode of the limit operator, we construct a so called {\it quasi-mode} and we are able to
prove that, for $\varepsilon$ sufficiently small, the diminution of the thickness of one rung of the ladder gives rise to
a localized mode. Moreover, diminushing $\varepsilon$, it is possible to create as many eigenvalues as one wants. We point out that the analysis of quantum graphs has been a very
active research area for the last three decades and we refer the reader to
the surveys
\cite{KuchmentQuantumgraphs1,KuchmentQuantumgraphs2,KuchmentQuantumSurvey}
as well as the books \cite{KuchmentBookGraph,PostBook} for an overview and an exhaustive bibliography of
this field.

\section{Presentation of the problem}
In the present work we study the propagation of waves in a ladder-like
periodic medium (see figure \ref{bande}). The homogeneous domain
$\Omega_{\varepsilon}$ (we will call it ladder) consists of the
infinite band of height $L$ minus an infinite set of equispaced
rectangular obstacles. The domain is $1$-periodic in one space
direction, corresponding to the variable $x$. The distance between two
consecutive obstacles is equal to the distance between the obstacles and
the boundary of the band and is denoted by $\varepsilon$.  

\begin{figure}[htbp]
	\begin{center}
	\subfloat[The
  unperturbed periodic ladder]{\label{bande}\includegraphics[width=0.5\textwidth, trim = 4cm 6cm 3cm 7cm, clip]{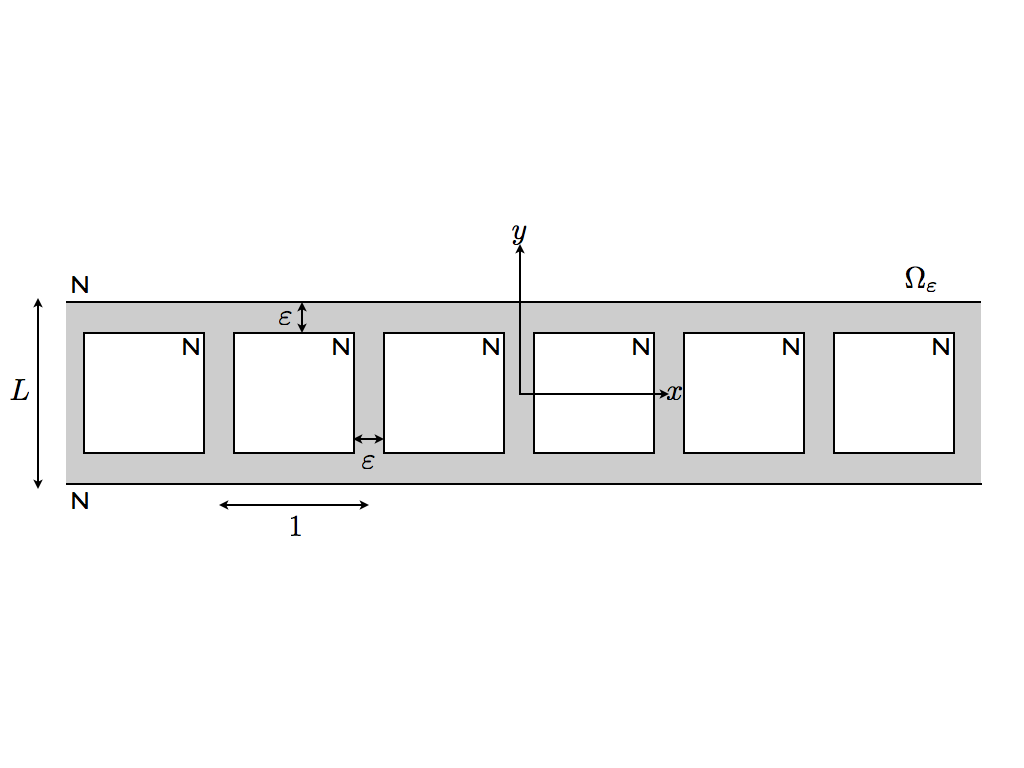}}\\
	\subfloat[The
  perturbed ladder]{\label{bande_mu}\includegraphics[width=0.5\textwidth, trim = 4cm 6cm 3cm 7cm, clip]{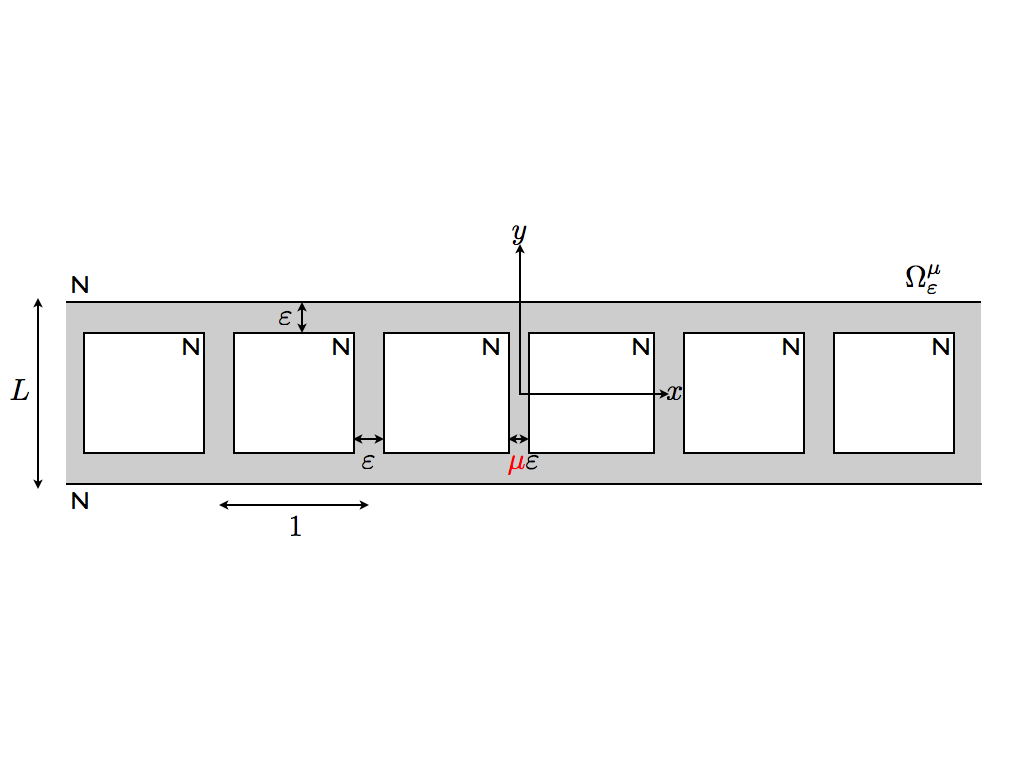}}	
  \caption{The unperturbed and the perturbed periodic ladders}   
\end{center}\end{figure}

\begin{rem}{Some extensions}
	We can change the distance between 2 consecutive obstacles from $\varepsilon$ to $\nu\varepsilon$. The study will depend on this new parameter but its conclusions remain the same.
	\end{rem}

The aim of this work is to find localized modes, that is solutions of the homogeneous scalar wave equation with Neumann boundary condition
\begin{equation}
\label{wave}
\frac{\partial^2v}{\partial{t}^2} -\Delta v = 0,\qquad \partial_n v=0\quad\text{on}\quad\partial\Omega_{\varepsilon},
\end{equation}
which are confined in the $x$-direction. 
\\[12pt]
Without giving a strict mathematical formulation (this will be done in the following section) a localized mode can be understood as a solution of the wave equation (\ref{wave}), which is harmonic in time
\begin{equation}
\label{wavesol1}
v(x,y,t)=u(x,y) \; e^{i\omega t},\qquad v\in L_2(\Omega_{\varepsilon}).
\end{equation}
where the function $u$ (which does not depend on time) belongs to $L_2$. The factor $e^{i\omega t}$ shows the harmonic dependence on time. Injecting (\ref{wavesol1}) into (\ref{wave}) leads to the following problem for the function $u$:
\begin{equation}
\label{eigproblem}
\begin{cases}
-\Delta u=\omega^2u\,&\text{ in }\Omega_{\varepsilon} \; ,\\
\displaystyle\partial_n u =0 & \text{ in }\partial \Omega_{\varepsilon}.
\end{cases}
\end{equation}
Problem (\ref{eigproblem}) is an eigenvalue problem posed in the
unbounded domain $\Omega_{\varepsilon}$. In order to create
eigenvalues, we introduce a local
perturbation in this perfectly periodic domain (the delicate question of existence of eigenvalues (or flat bands) for the unperturbed problem is not addressed in this paper, see for the absence of flat bands in waveguide problems (or the absolute continuity of the spectrum) for instance \cite{Sobolev:2002,friedlander:2003,Suslina:2002} and for the existence \cite{Filonov2001Absolute} ). The perturbed domain is obtained by changing the
distance between two consecutive obstacles from $\varepsilon$ to
$\mu\varepsilon$, $\mu>0$ (see
Figure \ref{bande_mu} in the case where $\mu< 1$). It
corresponds to modify the width of one vertical rung of the ladder
from $|x|< \varepsilon/2$ to $|x|< \mu \varepsilon/2$. 

As we will see such a perturbation does not change the continuous
spectrum of the underlying operator but it can introduce a non-empty
discrete spectrum. Our aim is to find eigenvalues by playing with the
values of 
$\mu$ and $\varepsilon$, $\varepsilon$ being treated as a small parameter.
\begin{rem}
	Imitating the approach developped in this article, it is also possible to study sufficient conditions which ensures the existence of guided modes in a ladder-like open periodic waveguides. More precisely, the domain is $\R^2$ minus an infinite set of equispaced perfect conductor rectangular obstacles with Neumann boundary conditions. And this domain is perturbed by a lineic defect, by changing the distance between two consecutive columns of obstacles. There exists a guided mode with a $\beta$ wave number if and only if there exists a localized mode in a perturbed periodic ladder where $\beta$-boundary conditions are imposed. The results of the present paper can be extended and the sufficient condition which ensures the existence of guided modes remains basically the same. 
	\end{rem}

\section{Mathematical formulation of the problem}
This section describes a mathematical framework for the analysis of the
spectral problem formulated above. We introduce the
operator $A_{\varepsilon}^{\mu}$, acting in the space
$L_2(\Omega_{\varepsilon}^{\mu})$, associated with the eigenvalue
problem (\ref{eigproblem}) in the perturbed domain:
\begin{equation*}
A^{{\mu}}_{\varepsilon}u=-\Delta u,\qquad
D(A^{{\mu}}_{\varepsilon})=\left\{u\in H^1_{\Delta}(\Omega_{\varepsilon}^{\mu}),\quad\partial_n u|_{\partial\Omega_{\varepsilon}^{\mu}}=0\right\}.
\end{equation*}
Here $\displaystyle H^1_{\Delta}(\Omega_{\varepsilon}^{\mu})=\left\{u\in H^1(\Omega_{\varepsilon}^{\mu}),\quad\Delta u\in L_2(\Omega_{\varepsilon}^{\mu})\right\}$.
The operator $A_{\varepsilon}^{\mu}$ is self-adjoint and positive. Our goal is to characterize its spectrum and, more precisely, to find sufficient conditions which ensures the existence of eigenvalues.  
\subsection{The essential spectrum of
  $A^{\mu}_{\varepsilon}$} To determine the essential spectrum of the
operator $A^{{\mu}}_{\varepsilon}$, we consider the case $\mu=1$,
where the domain $\Omega_{\varepsilon}$ is perfectly periodic (see Figure
\ref{bande}). We will denote the corresponding operator
$A_{\varepsilon}$. The Floquet-Bloch theory shows that the spectrum of periodic elliptic operators is reduced to its essential spectrum and  has a band-gap structure~\cite{Eastham:1973,Reed:1972,
  Kuchment:1993}:
\begin{equation}
\label{spectrum_intervals_general}
\sigma(A_{\varepsilon})=\sigma_{ess}(A_{\varepsilon})=\mathbb{R}\setminus\bigcup\limits_{1\leqslant n\leqslant N}]a_n,b_n[,
\end{equation}
where, in the previous formula, the union disappears if $N=0$.
For $N>0$, the intervals $]a_n,b_n[$ are called spectral gaps. Their number $N$
is conjectured to be finite (Bethe-Sommerfeld, 1933, \cite{Parnovski:2008,Parnovski:2010,Vorobets:2011}). The band-gap structure of the spectrum is a consequence of the following result given by the Floquet-Bloch theory:
\begin{equation}
\label{spectrum_union_general}
\sigma(A_{\varepsilon})=\bigcup\limits_{\theta\in[-\pi,\pi[}\sigma\left(A_{\varepsilon}(\theta)\right).
\end{equation}
Here $A_{\varepsilon}(\theta)$ is the Laplace operator defined on the
periodicity cell
$\mathcal{C}_{\varepsilon}=\Omega_{\varepsilon}\cap\big\{x\in(-1/2,1/2
\, )\big\}$ (see Figure \ref{cell}) with $\theta$-quasiperiodic
boundary conditions on the lateral boundaries $\Gamma_\varepsilon^\pm
= \partial \mathcal{C}_{\varepsilon} \cap \{x = \pm 1/2\}$ and
homogeneous Neumann boundary conditions on the remaining part
$\Gamma_\varepsilon = \partial \mathcal{C}_{\varepsilon} \, \setminus
(\Gamma_\varepsilon^- \cup \Gamma_\varepsilon^+ )$ of the boundary: for $\theta\in[-\pi,\pi]$,
\begin{align*}
&A_{\varepsilon}(\theta)\,:\;L_2(\mathcal{C}_{\varepsilon})\longrightarrow L_2(\mathcal{C}_{\varepsilon}),\qquad A_{\varepsilon}(\theta)u=-\Delta u,\\
&D\big(A_{\varepsilon}(\theta)\big)=\left\{u\in
  H^1_{\Delta}(\mathcal{C}_{\varepsilon}) ,\quad\partial_n
  u|_{\Gamma_\varepsilon}=0, \quad  u|_{\Gamma_\varepsilon^+}=e^{-i\theta} \; u|_{\Gamma_\varepsilon^-},\quad\partial_x u|_{\Gamma_\varepsilon^+}=e^{-i\theta} \; \partial_x u|_{\Gamma_\varepsilon^-}\right\}.
\end{align*}
\begin{center}
\begin{figure}[h]
\begin{center}
\includegraphics[width=0.4\textwidth, trim = 10cm 8cm 10cm 10cm,clip]{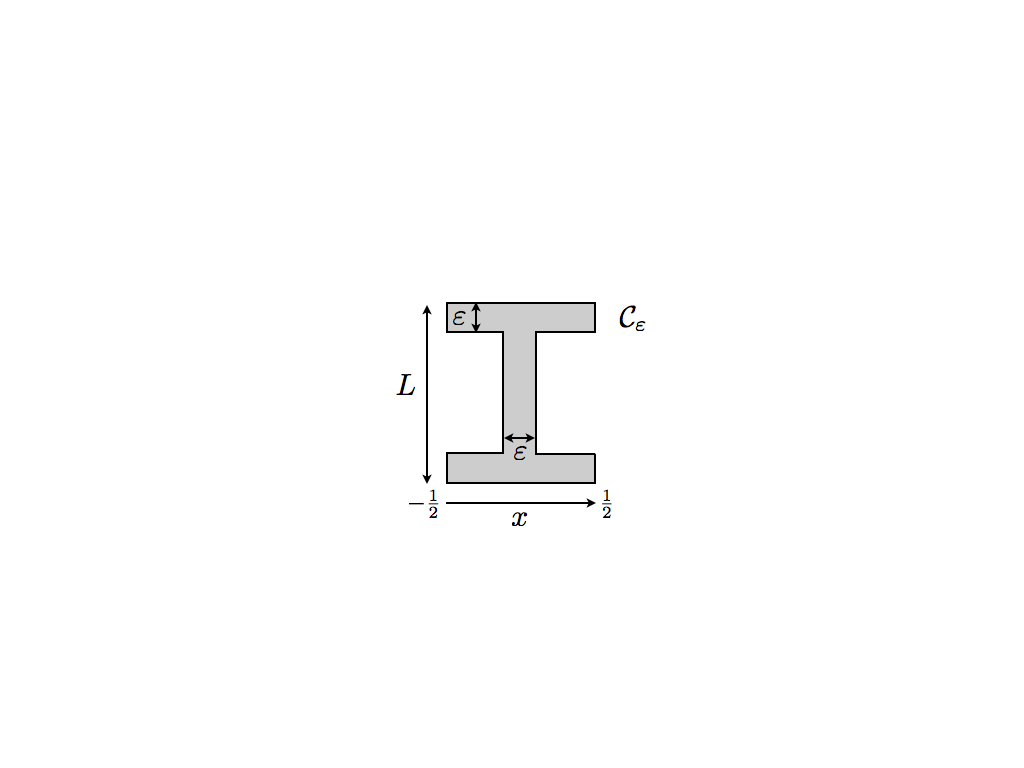}\caption{\label{cell}Periodicity cell}
\end{center}       
\end{figure}
\end{center}
For each $\theta\in[-\pi,\pi[$ the operator $A_{\varepsilon}(\theta)$
is self-adjoint, positive and its resolvent is compact. Its spectrum is then a sequence of non-negative eigenvalues of finite multiplicity tending to infinity:
\begin{equation}
\label{eigs}
0\leqslant\lambda^{(1)}_\varepsilon(\theta)\leqslant\lambda^{(2)}_\varepsilon(\theta)\leqslant\dots\leqslant\lambda^{(n)}_\varepsilon(\theta)\leqslant\dots,\qquad \lim\limits_{n\rightarrow\infty}\lambda^{(n)}_\varepsilon(\theta)=+\infty.
\end{equation} 
In (\ref{eigs}) the eigenvalues are repeated with their
multiplicity. The representative curves of the functions $\theta\mapsto\lambda_n(\varepsilon,\theta)$ are called dispersion curves and are known to be continuous and non-constant (cf. Theorem
XIII.86, volume IV in \cite{Reed:1972}). The fact that the dispersion curves are non-constant implies that the operator $A_{\varepsilon}$ has no eigenvalues of infinite multiplicity. Finally, (\ref{spectrum_union_general}) can be rewritten as
\begin{equation*}
\sigma(A_{\varepsilon})=\bigcup\limits_{n\in\mathbb{N}}\lambda^{(n)}_\varepsilon\big([-\pi,\pi]\big),
\end{equation*}
which gives (\ref{spectrum_intervals_general}). The conjecture of
Bethe-Sommerfeld means that for $n$ large enough the intervals
$\lambda_n\big(\varepsilon,[-\pi,\pi]\big)$ overlap or only touch. \\

\noindent Since
$D(A_{\varepsilon}(-\theta))=\overline{D(A_{\varepsilon}(\theta))}$
and the operators $A_\varepsilon$ have real coefficients, the function
$\lambda_n(\theta)$ are even.  Thus, it is sufficient to consider
$\theta\in[0,\pi]$ in (\ref{spectrum_union_general}).  This will be
used systematically in the rest of the paper.\\

 \noindent {As expected (this is related to Weyl's Theorem, see  in
 \cite[Chapter 13, Volume 4]{Reed:1972}, \cite[Chapter 9]{BirmanSolomjakBook} and
 \cite[Theorem 1]{Figotin:1997}), the essential spectrum is stable under a perturbation of the thickness
 of one rung of the ladder. In the present case, the domains of definition of the resolvents of  $A_{\varepsilon}^{\mu}$  and $A_{\varepsilon}^{\mu}$ are not the same. As a result, we cannot directly apply the standard results (in \cite[Chapter 9]{BirmanSolomjakBook}).}
\begin{prop}
\label{essential_spectrum_same}
$\sigma_{ess}(A_{\varepsilon}^{\mu})=\sigma_{ess}(A_{\varepsilon})$.
\end{prop}
\noindent {This  stability result is given in \cite[\S 4 Chapter 3, Theorem 4.1 Chapter 5]{nazarov1994elliptic} and ~\cite[Theorem 5]{NazarovPeriodic81}.  For the sake of completeness,  we provide a constructive proof based on the following assertion.}
\begin{lemme}
\label{lemme_chi}
Let $\chi\in C^{\infty}(\Omega_{\varepsilon})$ be a function such that 
\begin{enumerate}
\item[(a)]$\partial_n\chi|_{\partial\Omega_{\varepsilon}}=0$,
\item [(b)]$\exists \; M>0$ such that $\;|x|>M\;\Rightarrow\;\chi(x,y)=1$.
\end{enumerate}
If $\{u_j\}_{j\in\mathbb{N}}$ is a singular sequence for the operator $A_{\varepsilon}$ corresponding to the value $\lambda$, then there exists a subsequence of $\{\chi u_{j}\}_{{j}\in\mathbb{N}}$ which is also a singular sequence for the operator $A_{\varepsilon}$ corresponding to the value $\lambda$. 
\end{lemme}
\begin{proof}
By definition of a singular sequence, the sequence $\{u_j\}_{j\in\mathbb{N}}$ has the following properties:
\begin{enumerate}
\item $u_j\in D(A_{\varepsilon}),\quad j\in\mathbb{N}$;
\item $\inf\limits_{j\in\mathbb{N}}\|u_j\|_{L_2(\Omega_{\varepsilon})}>0$;
\item $u_j\stackrel{w}{\longrightarrow}0$\; in $\;L_2(\Omega_{\varepsilon})$;
\item $A_{\varepsilon} \, u_j-\lambda \, u_j\longrightarrow 0$ \; in $\;L_2(\Omega_{\varepsilon})$.
\end{enumerate}
Let us show that there exists a subsequence of $\{\chi u_{j}\}_{{j}\in\mathbb{N}}$ which has the same properties. The property 1 is verified by the whole sequence $\chi u_{j}$ thanks to property (a) . To prove property 2, it suffices to show that there exists a subsequence, still denoted $u_j$, such that
\begin{equation}
\label{help_proof_sp_ess}
\|u_j\|_{L_2(K_\varepsilon)}\longrightarrow 0,\quad j\longrightarrow\infty,\quad \text{  with }\; K_\varepsilon := \big\{ (x,y) \in \overline{\Omega_{\varepsilon}} \; / \; |x| \leq M \big\}. 
\end{equation}
Indeed, (\ref{help_proof_sp_ess}) and property (b) imply $\,\inf\limits_{j\in\mathbb{N}}\|\chi u_j\|_{L_2(\Omega_{\varepsilon})}\geqslant\inf\limits_{j\in\mathbb{N}}\|u_j\|_{L_2\left(\Omega_{\varepsilon}\cap\{|x|>M\}\right)}>0$. To prove (\ref{help_proof_sp_ess}), we write
\begin{equation*}
\|\nabla u_j\|^2_{L_2(\Omega_{\varepsilon})} = (A_{\varepsilon} \, u_j-\lambda \, u_j,u_j)_{L_2(\Omega_{\varepsilon})}+ \lambda \; \|u_j\|^2_{L_2(\Omega_{\varepsilon})}.
\end{equation*}
Then properties 3 and 4 imply that $u_j$ is bounded in $H^1(\Omega_{\varepsilon})$. By compactness, one can thus extract a subsequence that converges weakly $H^1(\Omega_{\varepsilon})$ and strongly in $L^2(K_\varepsilon)$ to a limit which is necessarily $0$ thanks to property 3, which proves  (\ref{help_proof_sp_ess}). For the sequel, we work with the above subsequence.
\\[12pt]   The property 3 being obvious the only thing to show is the property 4 for the sequence $\{\chi u_j\}_{j\in\mathbb{N}}$. We have: 
\begin{equation*}
\left\|A_{\varepsilon}(\chi u_j)-\lambda(\chi u_j)\right\|_{L_2(\Omega_{\varepsilon})}^2\leqslant\left\|\chi(A_{\varepsilon}u_j-\lambda u_j)\right\|_{L_2(\Omega_{\varepsilon})}^2+2 \, \|\nabla\chi \cdot \nabla u_j\|_{L_2(\Omega_{\varepsilon})}^2+\|u_j \, \Delta\chi\|_{L_2(\Omega_{\varepsilon})}^2.
\end{equation*}
The first and the last terms in the right-hand side tend to zero
thanks to property 4 and (\ref{help_proof_sp_ess})). Let us estimate
the second term. Using first  property (b), then properties (a) and (1)
together with an integration by parts, we obtain
\begin{align*}
\|\nabla\chi \cdot \nabla u_j\|_{L_2(\Omega_{\varepsilon})}^2=\int\limits_{K_{\varepsilon}}\nabla u_j  \cdot \nabla\overline{u}_j \, |\nabla\chi|^2 \, d{\bf x}=-\int\limits_{K_\varepsilon}u_j \,
\Big( |\nabla\chi|^2 \, \Delta \overline{u_j} + \nabla\big(|\nabla\chi|^2 ) \cdot \nabla \overline{u_j}\Big)\, d{\bf x} 
\end{align*}
which tends to $0$ due to (\ref{help_proof_sp_ess}) since $\nabla u_j$ is bounded in $L^2(\Omega_\varepsilon)$ as well as $\Delta u_j$ (by properties 3 and 4).  
\end{proof}
\begin{proof}[Proof of Proposition \ref{essential_spectrum_same}]
It is sufficient to take a function $\chi$ in the previous lemma which does not depend on $y$, vanishes in a neighbourhood of the perturbed edge and such that $\nabla\chi$ vanishes in a neighbourhood of all vertical edges. Then, it follows from Lemma \ref{lemme_chi} that any singular sequence associated to $\lambda$ of the operator $A_{\varepsilon}$ provides the construction of a singular sequence of the operator $A_{\varepsilon}^{\mu}$ for the same $\lambda$ and vice versa. 
\end{proof}
\noindent The essential spectrum of the operator $\sigma_{ess}(A_{\varepsilon}^{\mu})$ having a band-gap structure, we will be interested in finding eigenvalues inside gaps (once the existence of gaps is established).

\subsection{Towards the existence of eigenvalues: the method of study.} Our analysis consists of three main
steps. 
\begin{itemize} 
\item First, we find a formal limit of the eigenvalue problem (\ref{eigproblem}) when
$\varepsilon\rightarrow 0$ (Section~\ref{SectionDefinitionNotationAmu}). To do so, we use the fact that, as
$\varepsilon$ goes to zero, the domain 
$\Omega_{\varepsilon}^{\mu}$ shrinks to a graph ${\cal G}$. As a consequence, the
formal limit problem will involve a self-adjoint operator ${\cal A}^\mu$ associated with a 
second order differential operator along the graph. Its definition 
is strongly related to the fact that homogeneous Neumann boundary conditions
are considered in the original problem. 
More precisely, at the limit $\varepsilon \rightarrow 0$, looking for an eigenvalue
of $A_{\varepsilon}^\mu $ leads to search an eigenvalue of ${\cal A}^\mu$.
This operator, that is well known (see the works of \cite{Exner:1996,carlson:1998,KuchmentZeng:2001}), will be described
more rigorously in the next section.  
\item The second step is an explicit calculation of the spectrum of
  the limit operator. The essential spectrum is determined using the
  Floquet-Bloch theory (by solving a set of cell problems) (Section~\ref{SectionGrapheSym}) while the
  discrete spectrum of the perturbed operator is found using a
  reduction to a finite difference equation (Section~\ref{SubsectionSprectreDiscretAsmu}). In
  particular, we shall show that the limit operator has infinitely many
 eigenvalues of finite multplicity as
  soon as $\mu <1$ (and no one when $\mu \geq 1$), which form a
  discrete subset of $\R^+$. 
\item Finally, when $\mu < 1$, we deduce the existence of an eigenvalue of $A^\mu_\varepsilon$ close to the eigenvalue of ${\cal A}^\mu$ as soon as $\varepsilon$ is small enough (Section~\ref{sec:devasymp_ladder}). The
proof will be based on the construction of a quasi-mode (a kind approximation of the eigenfunction) and a criterion for
the existence of eigenvalues of self-adjoint operators (see for instance Lemma 4 in \cite{MR2766589}). It can be seen as a generalization of the well-known min-max principle for eigenvalues located below the lower bound of the 
essential spectrum. 
\end{itemize}
An essential preliminary step is the decomposition of the operator
$A_{\varepsilon}^{\mu}$ as the sum of two operators, namely its
symmetric and antisymmetric parts. To do so, we introduce the
following decomposition of $L_2(\Omega_{\varepsilon}^{\mu})$ :
\begin{equation*}
L_2(\Omega_{\varepsilon}^{\mu})=L_{2,s}(\Omega_{\varepsilon}^{\mu})\oplus L_{2,a}(\Omega_{\varepsilon}^{\mu}),
\end{equation*}
where $L_{2,s}(\Omega_{\varepsilon}^{\mu})$ and $L_{2,a}(\Omega_{\varepsilon}^{\mu})$ are subspaces consisting of functions respectively symmetric and antisymmetric with respect to the axis $y=0$:
\begin{equation*}
L_{2,s}(\Omega_{\varepsilon}^{\mu})=\left\{u\in
  L_{2}(\Omega_{\varepsilon}^{\mu})\,/\; u(x,y)=u(x,-y), \;\forall
  \,(x,y) \in \Omega_\varepsilon^\mu\right\},
\end{equation*}
\begin{equation*}
L_{2,a}(\Omega_{\varepsilon}^{\mu})=\left\{u\in
  L_{2}(\Omega_{\varepsilon}^{\mu})\,/\; u(x,y)=-u(x,-y), \; \forall \,(x,y) \in  \Omega_\varepsilon^\mu\right\}.
\end{equation*}
The operator $A_{\varepsilon}^{\mu}$ is then decomposed into the orthogonal sum
\begin{equation*}\label{decomposition_A}
A_{\varepsilon}^{\mu}=A_{\varepsilon,s}^{\mu}\oplus A_{\varepsilon,a}^{\mu}
, \quad 
A_{\varepsilon,s}^{\mu}=\left.A_{\varepsilon}^{\mu}\vphantom{A_{{L_{2,s}(\Omega_{\varepsilon}^{\mu})}}}\right\vert_{L_{2,s}(\Omega_{\varepsilon}^{\mu})}, \qquad A_{\varepsilon,a}^{\mu}=\left.A_{\varepsilon}^{\mu}\vphantom{A_{{L_{2,s}(\Omega_{\varepsilon}^{\mu})}}}\right\vert_{L_{2,a}(\Omega_{\varepsilon}^{\mu})}.
\end{equation*}
Accordingly, the limit operator ${\cal A}^\mu$ is decomposed as:
\begin{equation}
	{\cal A}^\mu = {\cal A}_s^\mu \oplus  {\cal A}_{a}^\mu
\end{equation}
The key point is that, as we shall see, contrary to the full operator
${\cal A}^\mu$ whose spectrum is $\Bbb R^+$, both operators ${\cal A}_s^\mu$ and ${\cal A}_{a}^\mu$ have spectral gaps (an infinity of them), each of them containing eigenvalues : these are isolated eigenvalues for ${\cal A}_s^\mu$ or ${\cal A}_{a}^\mu$, but embedded eigenvalues for
${\cal A}^\mu$. One deduces that the operators
$A_{\varepsilon,s}^{\mu}$ and $A_{\varepsilon,a}^{\mu}$ have  at
least finitely many spectral gaps, the number of gaps tending to $+\infty$ when $\varepsilon$ goes to 0: this is an important fact for applying the quasi-mode approach.\\

\noindent At this stage, it is worthwhile mentionning that the convergence of the spectrum of differential operators in thin
domains degenerating into a graph is not a new subject, particularly
in the case of elliptic operators. In particular, for the Laplace
operator with Neumann boundary conditions and in the case of compact
domains, the convergence results (which are reduced to the convergence
of eigenvalues) have been known since the works of
Rubinstein-Schatzman \cite{SchatzmannRubinstein:2001} and
Kuchment-Zheng \cite{KuchmentZeng:2001}. Thanks to the Floquet Bloch
theory, such results are transformed into analogous results for
thin periodic domains ({in~\cite[Theorem 5.1]{KuchmentZeng:2001}}), since in this case, only continuous spectrum is
involved. For general unbounded domains, a general (and somewhat
abstract) theory has been developed by Post in \cite{Post:2006} for
the convergence of all spectral components. This theory can be applied
to our problem, 
however, for the sake of simplicity, we have chosen to use here a more direct approach (based on the construction quasi-modes). 
\section{Spectral problem on the graph}\label{SectionEtudeDuGraphe}
\subsection{The operator ${\cal A}^{\mu}$.}\label{SectionDefinitionNotationAmu}
As $\varepsilon\rightarrow 0$, the domain $\Omega_{\varepsilon}$ tends
to the periodic graph $\cal G$ represented on Figure~\ref{graphe}. Let us
number the vertical edges of the graph $\cal G$ from left to right so that
the set of the vertical edges is $\{e_j = \{ j\} \times (-L/2, L/2) \}_{j\in\mathbb{Z}}$. The upper
end of the edge $e_j$ is denoted by $M_j^+$ and the lower one by
$M_j^-$. The set of all the vertices of the graphs is then
$$\mathcal{M}=\{M_j^{\pm}\}_{j\in\mathbb{Z}}.$$ The horizontal edge joining the
vertices $M_j^{\pm}$ and $M_{j+1}^{\pm}$  is denoted by
$e_{j+\frac{1}{2}}^{\pm} = (j, j+1) \times \{\pm L/2\}$. The set of all the edges of the graph is
$$\mathcal{E}=\{e_j, e_{j+\frac{1}{2}}^{\pm}\}_{j\in\mathbb{Z}}$$ and we denote by
$\mathcal{E}(M)$ the set of all the edges of the graph containing the vertex $M$.
\begin{center}
\begin{figure}[h]
\begin{center}
\includegraphics[width=0.75\textwidth, trim = 0cm 8cm 0cm 9cm,
clip]{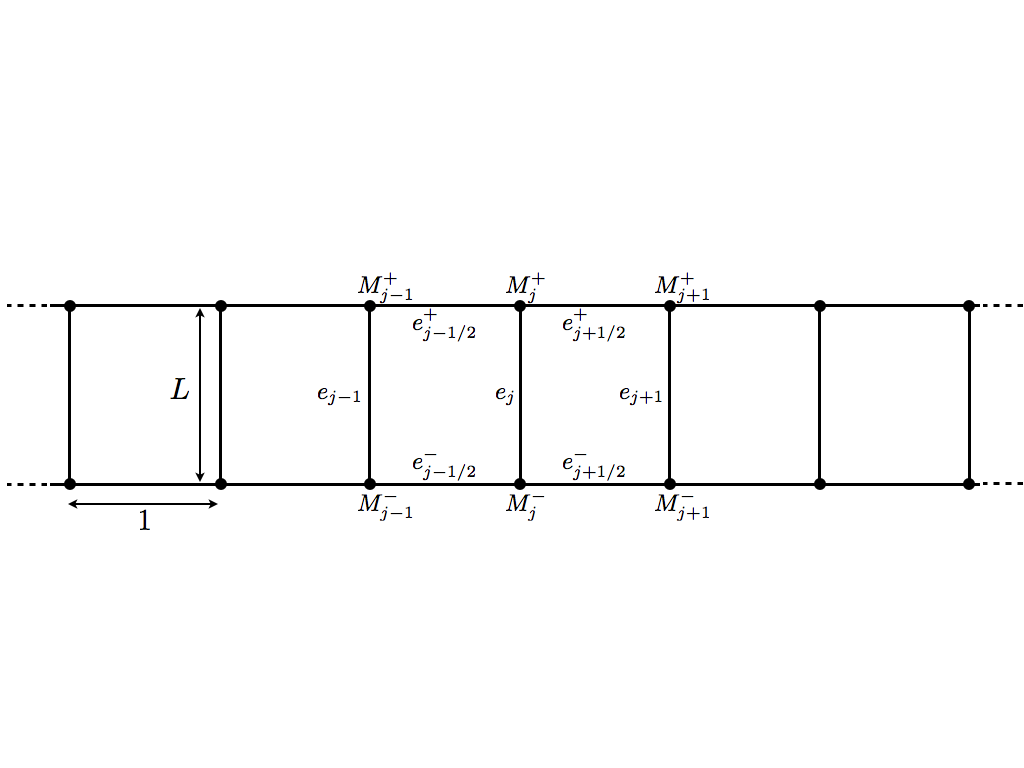}\caption{\label{graphe}Limit graph $\mathcal{G}$}
\end{center}       
\end{figure}
\end{center}
If $u$ is a function defined on $\cal G$ we will use the following
notation: $$\textbf{u}_j^{\pm}=u(M_j^{\pm}), \quad u_j(y)=u|_{e_j},
\quad {u}^{\pm}_{j+\frac{1}{2}}(x)=u|_{e_{j+\frac{1}{2}}^{\pm}}.$$
\\[12pt]
Let $w^{\mu}:\mathcal{E}\longrightarrow \mathbb{R}^+$ be a weight
function which is equal to $\mu$ on the "perturbed edge" $e_0$, i.e. the limit of the perturbed rung $|x|<\mu\varepsilon /2$, and to $1$ on the other edges:
\begin{equation}\label{eq:w_mu}
w^{\mu}(e_0)=\mu, \quad
w^{\mu}(e)=1,\quad \forall \; e\in \mathcal{E},\,e\neq e_0.
\end{equation}
Let us now introduce the following functional spaces
\begin{equation}
\label{L2G}
L_2^{\mu}(\mathcal{G})=\left\{u\:/\: u\in L_2(e),\: \forall e\in \mathcal{E}; \quad\|u\|^2_{L_2^{\mu}(\mathcal{G})}=\sum\limits_{e\in \mathcal{E}}w^{\mu}(e) \, \|u\|_{L_2(e)}^2<\infty \right\},
\end{equation}
\begin{equation*}
\label{H1G}
H^1(\mathcal{G})=\left\{u \in L_2^{\mu}(\mathcal{G})\:/\: u\in C(\mathcal{G}); \quad u\in H^1(e),\: \forall e\in \mathcal{E}; \quad\|u\|^2_{H^1(\mathcal{G})}=\sum\limits_{e\in \mathcal{E}}\|u\|_{H^1(e)}^2<\infty \right\},
\end{equation*}
\begin{equation}
\label{H2G}
H^2(\mathcal{G})=\left\{u \in L_2^{\mu}(\mathcal{G})\:/\: u\in C(\mathcal{G}); \quad u\in H^2(e),\: \forall e\in \mathcal{E}; \quad\|u\|^2_{H^2(\mathcal{G})}=\sum\limits_{e\in \mathcal{E}}\|u\|_{H^2(e)}^2<\infty \right\},
\end{equation}
where $C({\cal G})$ denotes the space of continuous functions on ${\cal G}.$:
\\\\We define the limit operator ${\cal A}^{\mu}$ in
$L_2^{\mu}(\mathcal{G})$ as follows. Denoting $u_e$ the restriction of
$u$ to $e$,
\begin{align}
&({\cal A}^{\mu}u)_{e}=-u''_{e}, \qquad\forall e\in \mathcal{E},\\[1ex]
&D({\cal A}^{\mu})=\big\{u\in H^2(\mathcal{G})\:/\:\sum\limits_{e\in
    \mathcal{E}(M)}w^{\mu}(e) u_{e}'(M)=0,\quad \forall M\in \mathcal{M}\big\},\label{Kirchoff_def}
\end{align}
where $u_{e}'(M)$ stands for the derivative of the function $u_e$
at the point $M$ in the \textit{outgoing direction}. The vertex
relations in (\ref{Kirchoff_def}) are called Kirchhoff's
conditions. Note that they all have an identical expression except at the vertices
$M_0^\pm$.
The following assertion as well as its proof can be found in
\cite[Section 3.3]{KuchmentQuantumgraphs1}. 
\begin{prop}[Kuchment]
The operator ${\cal A}^{\mu}$ in the space $L_2^{\mu}(\mathcal{G})$ is self-adjoint. The corresponding closed sesquilinear form has the following form:
\begin{equation*}
a^{\mu}[f,g]=(f',g')_{L_2^{\mu}(\mathcal{G})},\quad \forall f,g\in D[a^{\mu}],\qquad D[a^{\mu}]=H^1(\mathcal{G}).
\end{equation*}
\end{prop}
\noindent As for the ladder, we introduce the following decomposition of the space $L_2^{\mu}(\mathcal{G})$ into the spaces of symmetric and antisymmetric functions:
\begin{equation*}
L_{2}^{\mu}(\mathcal{G})=L_{2,s}^{\mu}(\mathcal{G})\oplus L_{2,a}^{\mu}(\mathcal{G}),
\end{equation*}
\begin{equation*}
L_{2,s}^{\mu}(\mathcal{G})=\left\{u\in L_2(\mathcal{G})\,/\;
  u(x,y)=u(x,-y), \;\forall\, (x,y) \in \mathcal{G}\right\},
\end{equation*}
\begin{equation*}
L_{2,a}^{\mu}(\mathcal{G})=\left\{u\in L_2(\mathcal{G})\,/\;
  u(x,y)=-u(x,-y), \; \forall \,(x,y) \in \mathcal{G}\right\}.
\end{equation*} 
Again, the operator ${\cal A}^{\mu}$ can be decomposed into the orthogonal sum
\begin{equation*}
{\cal A}^{\mu}={\cal A}_{s}^{\mu}\oplus {\cal A}_{a}^{\mu},
\end{equation*}
with
\begin{equation*}
{\cal A}_{s}^{\mu}=\left.{\cal A}^{\mu}\vphantom{A_{{L_{2,s}(\Omega_{\varepsilon}^{\mu})}}}\right\vert_{L_{2,s}^{\mu}(\mathcal{G})}, \qquad {\cal A}_{a}^{\mu}=\left.{\cal A}^{\mu}\vphantom{A_{{L_{2,s}(\Omega_{\varepsilon}^{\mu})}}}\right\vert_{L_{2,a}^{\mu}(\mathcal{G})},
\end{equation*}
which implies
\begin{equation*}
\sigma({\cal A}^{\mu})=\sigma({\cal A}_{s}^{\mu})\cup\sigma({\cal A}_{a}^{\mu}).
\end{equation*}
Thus, it is sufficient to study the spectrum of the operators ${\cal
  A}_{s}^{\mu}$ and ${\cal A}_{a}^{\mu}$ separately. The analysis of
these two operators being analogous, we will present a detailed study
of ${\cal A}_{s}^{\mu}$ (Section~ \ref{SectionGrapheSym} and
Section~ \ref{SubsectionSprectreDiscretAsmu}) and state the results for ${\cal A}_{a}^{\mu}$ (Section~ \ref{SectionSpectreAntisymetrique}).
\\\\ 
\subsection{The essential spectrum of the operator
  ${\cal A}_{s}^\mu$}\label{SectionGrapheSym}
We shall study the spectrum of the operator ${\cal A}_s^\mu$ by a
perturbation technique with respect to the case $\mu=1$ which
corresponds to the purely periodic case. The corresponding operator
will be denoted by ${\cal A}_s$.  
Indeed, based on compact perturbation arguments {( in~\cite[Theorem 4, Chapter 9]{BirmanSolomjakBook})}, we
can prove the following proposition:
\begin{prop}
\label{proposition_spectre} The essential spectra of ${\cal A}_s^{\mu}$ and
${\cal A}_s$ coincide:
\begin{equation}
 \sigma_{ess}({\cal A}_s^{\mu})=\sigma_{ess}({\cal A}_s).
\end{equation}
\end{prop}
\noindent This reduces the study of the essential spectrum of ${\cal A}_s^\mu$ to the
study of the spectrum of the purely periodic operator ${\cal A}_s$, which can
be done through the Floquet-Bloch theory.

\subsubsection{Description of the spectrum of ${\cal A}_s$ through Floquet-Bloch theory}  
As previously explained, the spectrum of the operator ${\cal A}_{s}$
can be studied using the Floquet-Bloch theory. One has then to study a
set of problems set on the periodicity cell of $\mathcal{G}$. Since
we consider the subspace of symmetric functions with respect to the
axis $y=0$, this enables to reduce the problem to the lower half part
of the periodicity cell $\mathcal{C}_{-}=\mathcal{G}\cap
[-1/2,1/2]\times [-L/2,0]$ (see Figure~\ref{halfcell_graph}).
\begin{center}
\begin{figure}[htbp]
\begin{center}
\includegraphics[width=0.3\textwidth, trim=5cm 9cm
  16cm 10cm, clip]{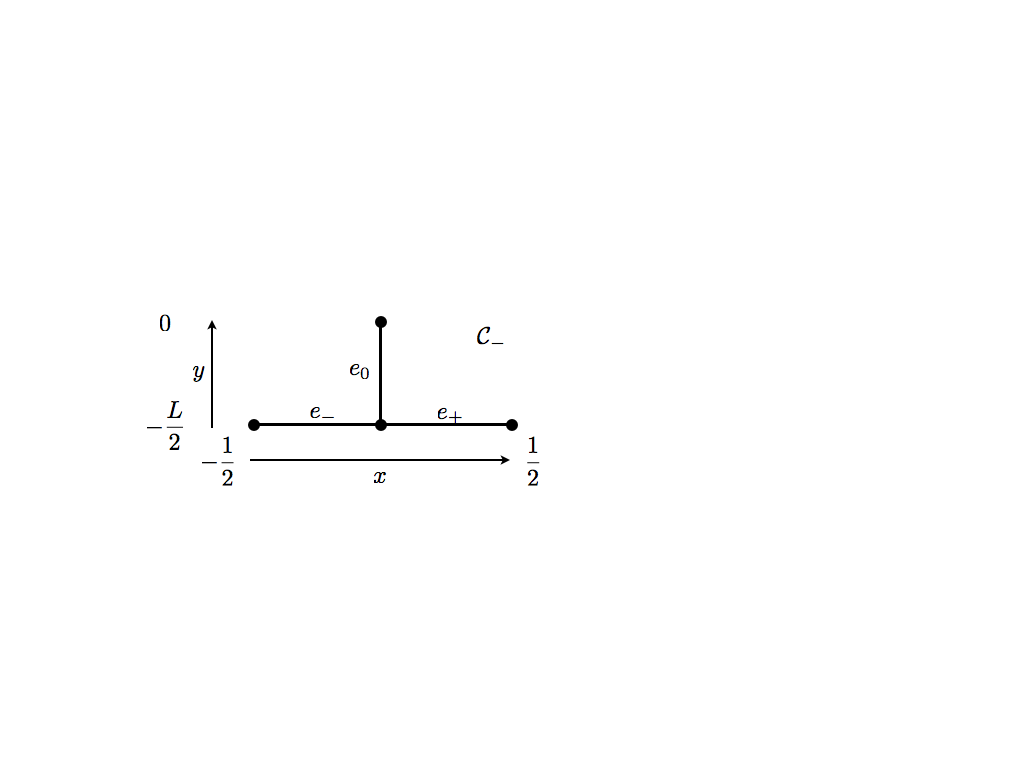}\caption{\label{halfcell_graph}Periodicity cell}
\end{center}       
\end{figure}
\end{center}
We introduce the spaces $L_2(\mathcal{C}_-)$ and $H^2(\mathcal{C}_-)$ analogously to (\ref{L2G}), (\ref{H2G}):
\begin{equation*}
L_2(\mathcal{C}_-)=\big\{u\:/\: u_\ell := u|_{e_\ell} \in L_2(e_\ell),
\,\ell \in \{0, +, -\}\big\},
\end{equation*}
\begin{equation*}
H^2(\mathcal{C}_-)=\big\{u\in C(\mathcal{C}_-) / u_\ell\in
H^2(e_\ell), \ell \in \{0, +, -\}\big\}.
\end{equation*}
We have then
\begin{equation}
\label{spectrum}
\sigma({\cal A}_{s})=\bigcup\limits_{\theta\in[0,\pi]}\sigma\left({\cal A}_{s}(\theta)\right)
\end{equation}
where ${\cal A}_{s}(\theta)$ is the following unbounded operator in $L_2(\mathcal{C}_-)$
$$
\begin{array}{|l}
\dsp \big[{\cal A}_{s}(\theta)u\big]_\ell=-u_\ell'', \; \ell \in \{0, +, -\}, \\[6pt]
\dsp D({\cal A}_{s}(\theta))=\left\{u\in H^2(\mathcal{C}_-)\,/ \; u \mbox{ satisfies } (\ref{kirchoff}) \mbox{ and } u_0'(0)=0 \right\}
\end{array}
$$
\begin{equation} \label{kirchoff}
	\begin{array}{c}
(a) \quad  u'_+(0) -u'_-(0) +u'_0\left(-{L}/{2}\right)=0, \\[5pt] (b) \quad 
u_+\left({1}/{2}\right)=e^{-i\theta}u_-\left(-{1}/{2}\right),\quad u'_+\left({1}/{2}\right)=e^{-i\theta}u'_-\left(-{1}/{2}\right).
\end{array}
\end{equation}
In the definition of $D\big({\cal A}_{s}(\theta)\big)$, the condition $u_0'(0)=0$ corresponds to the symmetry with respect to $y=0$, (\ref{kirchoff})-(a) is the Kirchhoff's condition with $\mu=1$ and 
(\ref{kirchoff})-(b) are the $\theta$-quasiperiodicity conditions.
For each $\theta\in[0,\pi]$, the operator ${\cal A}_{s}(\theta)$ is self-adjoint and positive and its resolvent is compact due to the compactness of the embedding $H^1(\mathcal{C}_{-})\subset L_2(\mathcal{C}_{-})$. Consequently, its spectrum is a sequence of non-negative eigenvalues of finite multiplicity tending to infinity:
\begin{equation}
\label{eig}
0\leqslant\lambda^{(1)}_{s}(\theta)\leqslant\lambda^{(2)}_{s}(\theta)\leqslant\dots\leqslant\lambda^{(n)}_{s}(\theta)\leqslant\dots,\qquad \lim\limits_{n\rightarrow\infty}\lambda^{(n)}_{s}(\theta)=+\infty.
\end{equation} 
In the present case, the eigenvalues can be computed explicitly. 
\begin{prop}
\label{eigenvalues_sym}
For $\theta\in[0,\pi]$, $\omega^2\in\sigma({\cal A}_s(\theta))$ if and only if $\omega$ is a solution of the equation
\begin{equation}
\label{caract_spectre_ess_sym}
\displaystyle 2\cos{({\omega L}/{2})}\left(\cos{\omega}-\cos{\theta}\right)=\sin{\omega} \, \sin{({\omega L}/{2})}.
\end{equation}
\end{prop}
\begin{proof}
If $\omega^2\neq 0$ is an eigenvalue of the operator ${\cal A}_s(\theta)$ then the corresponding eigenfunction $u=\{u_0, u_+, u_-\}$ is of the form   
\begin{align}
 u_-(x)&=a_- \, e^{i\omega x}+b_- \, e^{-i\omega x},\qquad x\in[-1/2,0],\label{u1}\\
 u_+(x)&=a_+ \, e^{i\omega x}+b_+ \, e^{-i\omega x},\qquad x\in[0,1/2],\\
 u_0(y)&=a_0 \, e^{i\omega y}+b_0 \, e^{-i\omega y},\qquad y\in[-L/2,0].\label{u3}
\end{align}
Taking into account that $u\in D({\cal A}_s(\theta))$, we arrive at the following linear system 
\begin{align} 
&a_-+b_-=a_++b_+=a_0 \, e^{-i\omega L/2}+b_0 \, e^{i\omega L/2},\label{continuity}\\
&a_0=b_0,\label{symmetry}\\
&b_- -a_- +a_+-b_++a_0 \, e^{-i\omega L/2}-b_0 \, e^{i\omega L/2}=0\label{kirch}, \\
&a_+e^{i\omega/2}+b_+e^{-i\omega/2}=e^{-i\theta}\big(a_-e^{-i\omega/2}+b_-e^{i\omega/2}\big),\label{qpfun}\\
&a_+e^{i\omega/2}-b_+e^{-i\omega/2}=e^{-i\theta}\big(a_-e^{-i\omega/2}-b_-e^{i\omega/2}\big)\label{qpderiv}. 
\end{align}
The relations (\ref{continuity}) express the continuity of the eigenfunction at the vertex $(0, -L/2)$. The equation (\ref{symmetry}) comes from the condition $u'_0(0)=0$. The relation (\ref{kirch}) corresponds to (\ref{kirchoff})-(a) while (\ref{qpfun}) and (\ref{qpderiv}) correspond to (\ref{kirchoff})-(b).
Adding and substracting (\ref{qpfun}) and (\ref{qpderiv}) lead to
$
a_-=a_+ \, e^{i (\theta+ \omega) } \,  \mbox{ and } b_-=b_+ \, e^{i (\theta- \omega) } , 
$
which we can substitute into (\ref{continuity})--(\ref{kirch}) to obtain the following system in $(a_+, b_+, a_0)$
\begin{equation}
	M(\theta, \omega, L) \; \begin{pmatrix} a_+ \\ b_+ \\ a_0 \end{pmatrix} = 0 \quad \mbox{where} \quad M(\theta, \omega, L) := \begin{pmatrix}  \; 1-e^{i (\theta+ \omega) }  & 1-e^{i (\theta- \omega) }  & 0\\[2ex]
\; 	1 & 1 & - 2 \cos(\frac{\omega L}{2})\\[2ex]
\; 	 1- e^{i (\theta+ \omega) }  &-1+e^{i (\theta- \omega) }  & -
2 i \sin(\frac{\omega L}{2}) \end{pmatrix} 
\end{equation}
It is then easy to conclude since one obtains, after some computations omitted here
$$
\mbox{det } M(\theta, \omega, L)  = 4
e^{i\theta} \, \left(2 \cos{\left(\textstyle\frac{\omega L}{2}\right)}\left(\cos{\omega}-\cos{\theta}\right) - \sin{\omega}\sin{\left(\textstyle\frac{\omega L}{2}\right)} \right).  $$
 For $\omega=0$, the relations (\ref{u1})--(\ref{u3}) are replaced by
\begin{align*}
 u_-(x)&=a_-+b_- \, x,\qquad x\in[-1/2,0],\\
 u_+(x)&=a_++b_+ \, x,\qquad x\in[0,1/2],\\
 u_0(y)&=a_0+b_0, y,\qquad y\in[-L/2,0].
\end{align*}
Using the fact that $u\in D({\cal A}_s(\theta))$ we have (instead of (\ref{continuity}-\ref{qpderiv})):
\begin{equation*} 
a_-=a_+=a_0,\qquad b_0=0,\qquad b_-=b_+, \qquad b_+=b_- \, e^{-i\theta},\qquad b_+=a_- \, (e^{-i\theta}-1). 
\end{equation*}
One then easily sees that there exists a non-trivial solution if and only if $\theta=0$ and that the corresponding eigenfunction is constant. Noticing that, for $\theta = 0$, $\omega=0$ is solution of (\ref{caract_spectre_ess_sym}) allows us to conclude.
\end{proof}
\noindent The reader will notice that when $L \in \mathbb{Q}$, the spectrum of ${\cal A}_s(\theta)$ has a particular structure: it is the image by the function $x \mapsto x^2$ of a periodic countable subset of $\mathbb{R}$. To see that, it suffices to remark that both functions at the left and right hand sides of (\ref{caract_spectre_ess_sym}) are periodic with a common period. As a consequence of (\ref{spectrum}), the spectrum of ${\cal A}_s(\theta)$  is the image by the function $x \mapsto x^2$ of a periodic  subset of $\mathbb{R}$.
\subsubsection{Characterization of the spectrum of ${\cal A}_s$}
\noindent Using (\ref{spectrum}), Proposition \ref{eigenvalues_sym} allows us to describe the structure of the spectrum of the operator ${\cal A}_s$. We first prove the existence of a countable infinity of gaps.
\begin{prop}\label{result_sym} The following properties hold
\begin{enumerate}
\item $\sigma_{2,s} \cup \sigma_{L,s}\subset\sigma({\cal A}_s) $, where  
$\sigma_{2,s} = \left\{(\pi n)^2,\; n\in\mathbb{N}\right\}$ and $\sigma_{L,s} = \left\{(2\pi n/L)^2,\; n\in\mathbb{N}\right\}$.

\item The operator ${\cal A}_s$ has infinitely many gaps whose ends tend to infinity.
\end{enumerate}

\end{prop} 
\begin{proof}\begin{enumerate}
	\item For $\sin\omega=0$ or $\sin(\omega L/2)=0$, the equation
          (\ref{caract_spectre_ess_sym}) is satisfied for
          $\cos(\theta)=\cos(\omega)$ so that $\omega$ belongs to  $\sigma({\cal A}_s(\theta))\subset\sigma({\cal A}_s)$.
\item Let $\omega_n=(2n+1)\pi/L$ (such that $\cos{(\omega_n L/2)}=0$), let us distinguish two cases:\\\\
(a) $\sin(\omega_n)\neq0$: the left hand side of equation
(\ref{caract_spectre_ess_sym}) vanishes for all $\theta$ and, as $\sin{(\omega_n L/2)}\neq 0$,  the right hand side does not. Then $\omega_n^2$ does not belong to the spectrum of $\sigma({\cal A}_s)$. Since $(2\pi n/L)^2$ and $(2\pi (n+1)/L)^2$ belong to $\sigma({\cal A}_s)$ (in view of the point 1) there exists a gap which contains $\omega_n^2$, strictly included in $((2\pi n/L)^2,(2\pi (n+1)/L)^2)$.\\\\
(b) $\sin(\omega_n)=0$: (this case can occur only for special values of $L$, see remark \ref{rem:Qc}), we know by point 1, that $\omega^2_n\in\sigma({\cal A}_s)$ and we are going to show that it exists $\delta>0$ such that
$((\omega_n-\delta)^2,\omega_n^2)$ and $(\omega_n^2,(\omega_n+\delta)^2)$ are in the resolvant set of ${\cal A}_s$. This will show the existence of two disjoint gaps of the form $(\omega_n^2-l_n^-,\omega_n^2)\subset ((2\pi n/L)^2,\omega_n^2)$ and $(\omega_n^2,\omega_n^2+l_n^+)\subset (\omega_n^2,(2\pi (n+1)/L)^2)$. Setting $\omega = \omega_n+z$ in relation (\ref{caract_spectre_ess_sym}) leads to
\begin{equation}\label{help_proof}
f_n(z)=0\quad\text{where}\quad f_n(z):=2\sin{(z L/2)}(\cos{\omega_n}\cos{z}-\cos \theta)+\cos{\omega_n}\sin{z}\cos{(z L/2)} \; .
\end{equation}
We have 
\[
	f_n(z) = z(\cos\omega_n+L(\cos\omega_n-\cos\theta))+o(z)
\]
which cannot vanish for $0<|z|<\delta$ for $\delta$ small enough,
since $\cos\omega_n=\pm 1$. This implies that $\cos\omega_n+L(\cos\omega_n-\cos\theta)\neq0$ for all $\theta$.
\\\\
The conclusion follows from the fact that the intervals $((2\pi
n/L)^2,(2\pi (n+1)/L)^2)$ are disjoint, go to infinity with $n$,
and contain one or two gaps.
\end{enumerate}
\end{proof}
\begin{rem}\label{rem:Qc}
	The case 2.(b) of the above proof can occur only for special values of $L$. Indeed, the reader will easily verify that the existence of $\omega$ such that $\sin(\omega)=cos(\omega L/2)=0$ is equivalent to the fact that 
	\begin{equation} \label{defQc}
	L\in\mathbb{Q}_c := \big\{ q \in \mathbb{Q} \; / \; \exists \; (m,k)	\in \mathbb{N} \times \mathbb{N}^* \mbox{ such that } q = \frac{2m+1}{k} \mbox{ (irreducible fraction)} \big\}.
	\end{equation}
	In fact, the condition \eqref{defQc} also influences the nature of the spectrum of ${\cal A}_s$. Indeed it can be shown that when $L$ does not belong to $\mathbb{Q}_c$, the point spectrum of ${\cal A}_s$ is empty (i. e. the spectrum of ${\cal A}_s$ is purely continuous). When $L$ belongs to $\mathbb{Q}_c$, it coincides with an infinity of eigenvalues of infinite multiplicity, associated with compactly supported eigenfunctions.
	It is worth noting that the presence of such eigenvalues is a
	specific feature of periodic graphs (see 
	\cite[Section 5]{KuchmentQuantumgraphs2}). 
\end{rem}
\begin{rem}
 In the
proof, in the case 2.(a), gaps are located in the
vicinity of the points $\lambda$ satisfying $\cos{(\lambda L/2)}=0$. These
points are nothing else but the eigenvalues of the 1d Laplace operator
defined on the vertical half edges $ \{(x,y), x = j ,  -L/2< y < 0\}$
with Dirichlet boundary condition at $y=-L/2$ and Neumann boundary
condition at  $y=0$. The presence of gaps is therefore consistent with
\cite[Theorem 5]{KuchmentQuantumgraphs2} dealing with gaps created by 
so-called graph decorations. Indeed, the vertical half edges can be
seen as
decorations of the infinite periodic graph $\mathcal{G}_0$ made of the set of the horizontal edges $e_{j+1/2}^+$.  
\end{rem}
\noindent Next, we give a more precise description of the gap structure of $\sigma({\cal A}_s)$ through a geometrical interpretation of \eqref{caract_spectre_ess_sym}. 
We first remark that as soon as $\omega \notin\{\pi\mathbb{Z}\}\cup\{2\pi\mathbb{Z}/L\}$, $\lambda = \omega^2 \mbox{ belongs to } \sigma({\cal A}_s)$ (i. e. $\omega$ is solution of \eqref{caract_spectre_ess_sym}) if and only if
\begin{equation}
\label{phi_f}
\exists \; \theta\in[0,\pi] \quad \mbox{such that} \quad \cos{\theta}
\neq \cos{\omega} \quad \mbox{ and } \; \phi_L(\omega)=f({\theta},\omega),
\end{equation}
where the functions $\phi_L$ and $f$ are defined by 
\begin{equation}\label{def_phi_f} 
\quad\phi_L(\omega):=\frac{2}{\tan{(\omega L/2)}}, \quad f({\theta},\omega):=\frac{\sin{\omega}}{\cos{\omega}-\cos{\theta}}.
\end{equation}
In the following we reason in the $(\omega,y)$-plane  with $y$ an additional auxiliary variable. We introduce the domain $D$
\begin{equation} \label{domainD}
\displaystyle	D = \left \{ \big(\omega, f({\theta},\omega)\big),
  (\omega, \theta) \in \R \times [0,\pi] \,\mbox{and} \,  \cos{\omega} \neq \cos{\theta}
\right \}.
\end{equation}
\begin{lemme} The domain $D$ is the domain of the $(\omega,y)$-plane,
  $\pi$-periodic with respect to $\omega$, given by
	\begin{equation}
D = \bigcup_{n \in \mathbb{Z}}	\left\{ D_0 + (n \pi,0) \right\},\; \quad	D_0 :=  [0, \pi] \times \mathbb{R} \setminus \big\{ (\omega,y) \; / \; 0 < \omega < \pi , -\tan(\frac{\omega}{2})^{-1}< y < \tan(\frac{\omega}{2}) \big\}.	
	\end{equation}
%
\end{lemme}
\begin{proof}
The $\pi$-periodicity of the domain $D$ with respect to $\omega$
follows from the identity $f(\theta, \omega + \pi) = f( \pi - \theta,
\omega)$. To conclude, it suffices to remark that, for a given $\omega \in (0,
\pi)$, if $\theta $
varies in the interval $[0, \omega)$ $\theta \mapsto f(\theta,
\omega)$ is continuous and strictly decreasing from $-
\left(\tan(\omega/2)\right)^{-1}$ to $- \infty$ while,  if $\theta$
varies in the interval $(\omega, \pi]$, $\theta \mapsto f(\theta,
\omega)$ is continuous and strictly decreasing from $+\infty$ to
$\tan({\omega}/{2})$.
\end{proof} 
\noindent Thanks to Proposition~\ref{result_sym} and the characterization~\eqref{phi_f}, we have
\begin{equation}\label{geometricalCharacterization}
\sigma({\cal A}_s) = \sigma_{2,s} \cup \sigma_{L,s} \cup \left\{ \omega^2
  \notin \sigma_{L,s} \, / \, (\omega,\phi_L(\omega)) \in D\right\}.
\end{equation}
\begin{figure}[htbp]
\centering
\subfloat{\label{Explication1-b}\includegraphics[width=0.45\textwidth, trim=7cm 6cm
  7cm 6cm, clip]{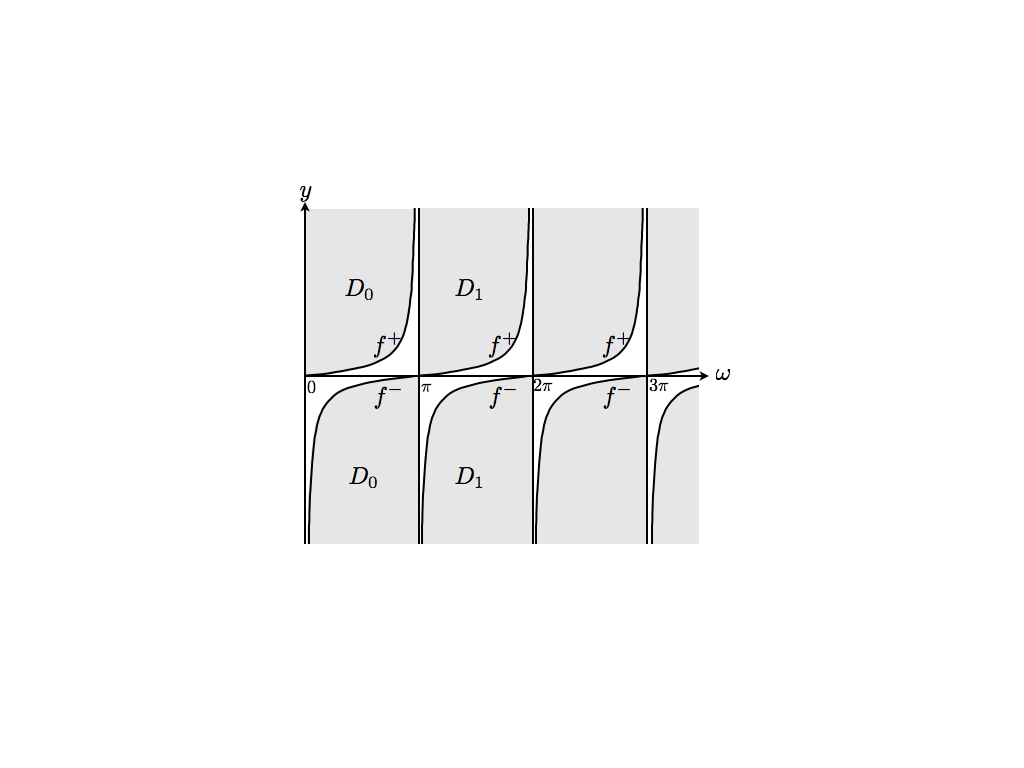}}
\subfloat{\label{Explication1-c}\includegraphics[width=0.45\textwidth,trim=7cm 6cm
  7cm 6cm, clip]{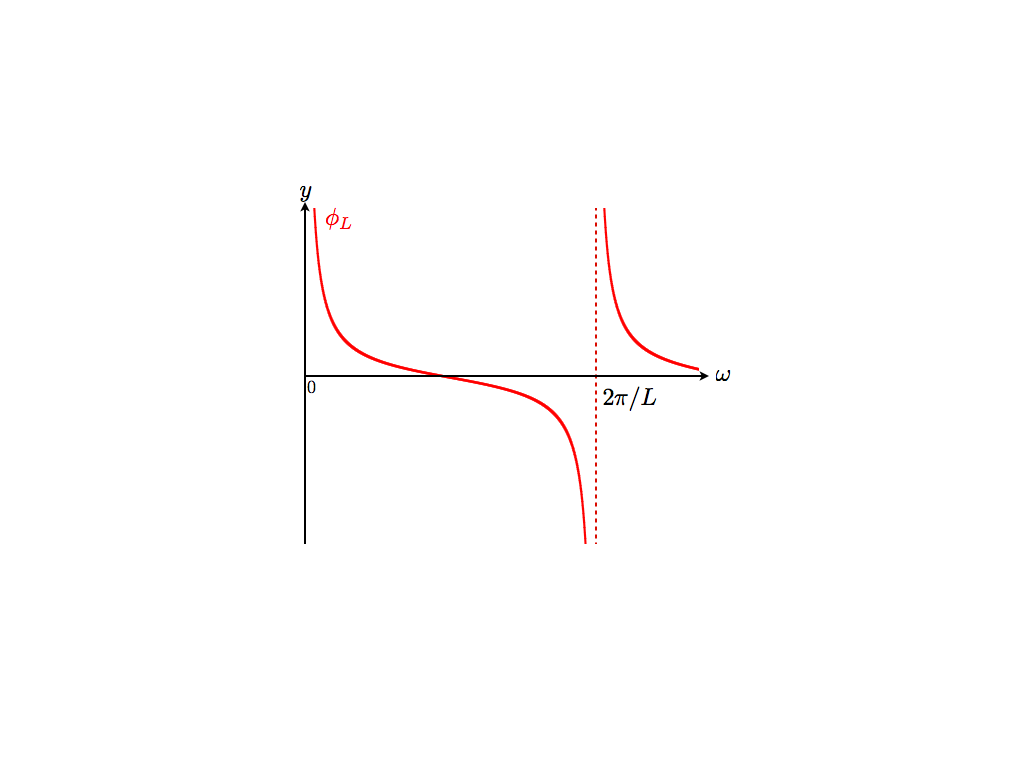}}
\caption{\label{Explication1}Representation of the $D$ (grey part) and the
  curve $\mathcal{C}_L$ (for $L=8$).}
\end{figure}
\begin{figure}[htbp]
\centering
\subfloat[$L=8$]{\label{Explication2-b}\includegraphics[width=0.45\textwidth, trim=7.5cm 6cm
  5.3cm 6cm, clip]{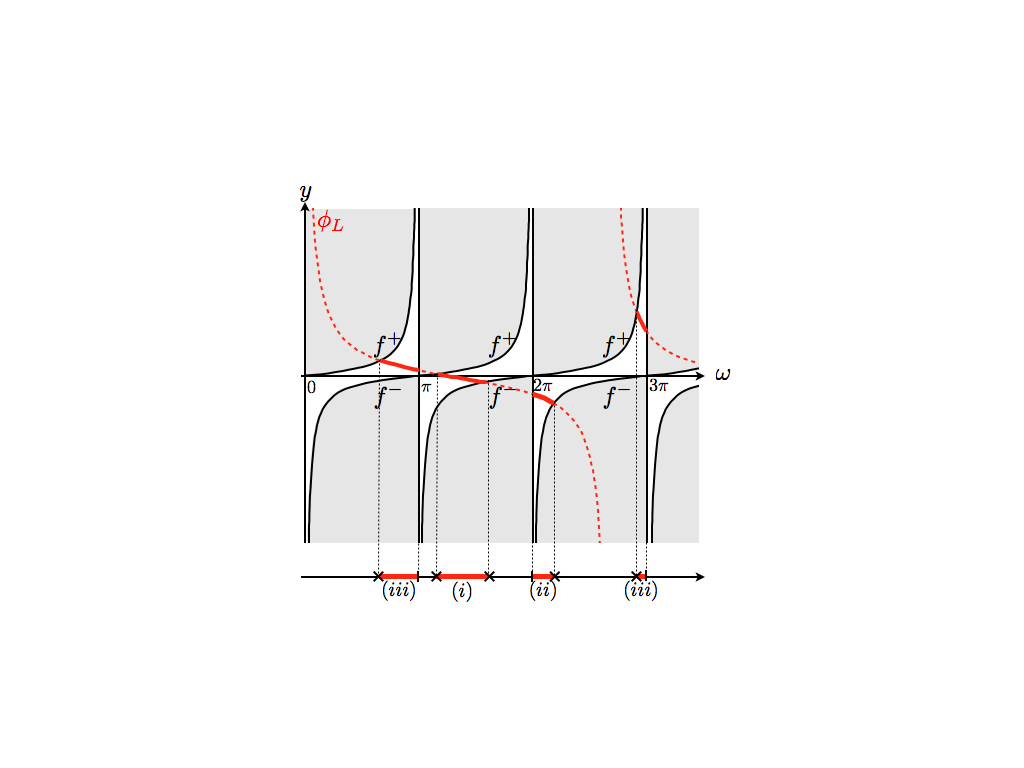}}
\subfloat[$L=10 \pi/7$]{\label{Explication2-c}\includegraphics[width=0.45\textwidth,trim=7.5cm 6cm
  5.3cm 6cm, clip]{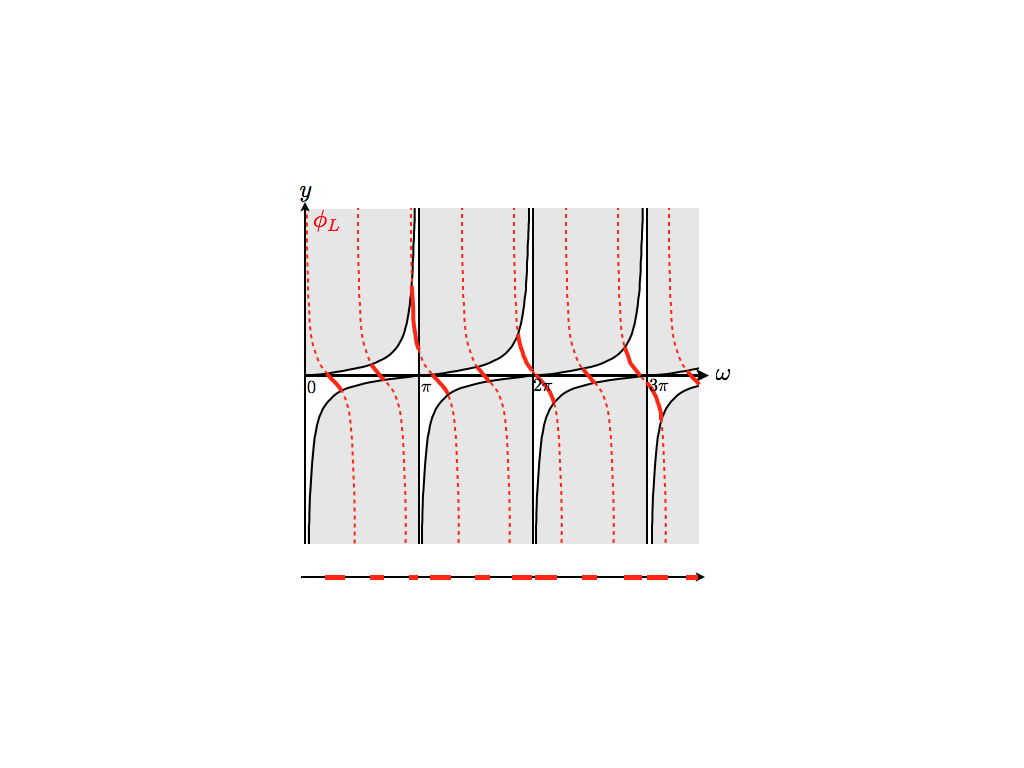}}
\caption{The images of the spectral gaps by $x \mapsto \sqrt{x}$. In
 the left picture, the three types of gaps are distinguished
 (according to the legend). \label{Explication2}}
\end{figure}
 \begin{figure}[htbp]
 \begin{center}
 \includegraphics[scale=0.4, trim = 4cm 6cm 3cm 7cm, clip]{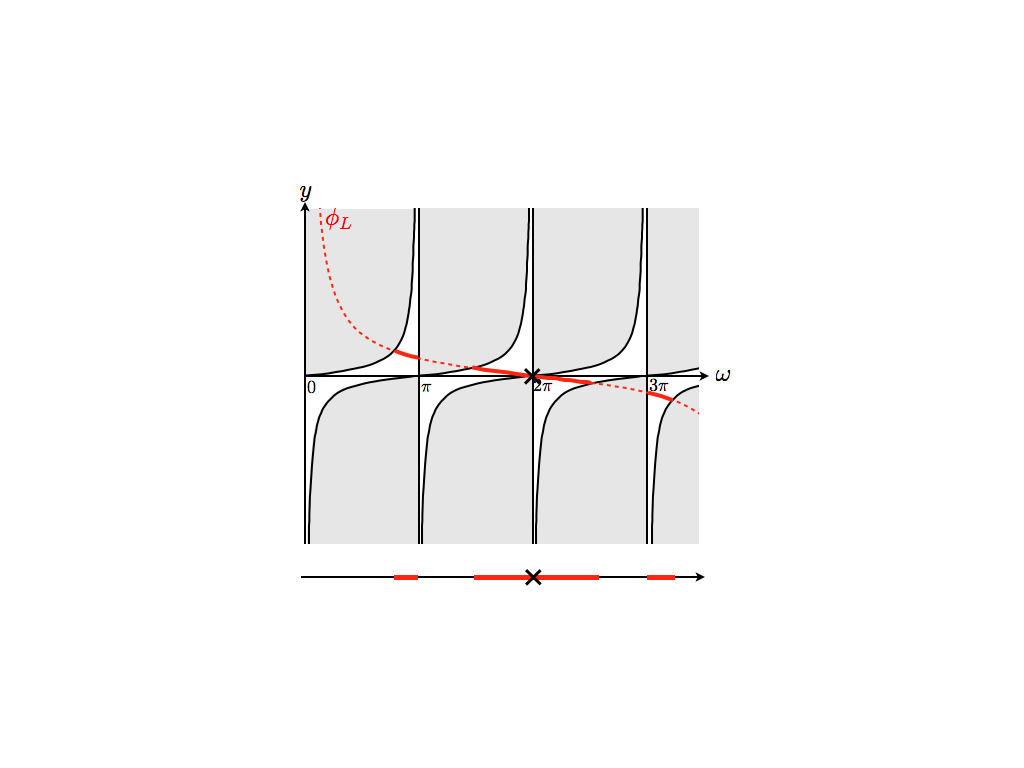}
 \caption{An example of eigenvalue of infinite
     multiplicity ($\omega = 2 \pi$)  obtained for $L=1/2$.  This eigenvalue separates a gap of type (ii) on
     the left from a gap of type (iii) on the right.  This occurs for
     $L \in \mathbb{Q}_c$.}
\label{Particular cases}
 \end{center}
 \end{figure}
\noindent In other words, $\sigma({\cal A}_s)$ is the union of $\sigma_{2,s}$,
$\sigma_{L,s}$, and  the image by the application $x \mapsto x^2$ of the projection on the
line $y=0$ of the intersection of the  domain
$D$ with the curve $\mathcal{C}_L=\left\{ (\omega, \phi_L(\omega)),
  \omega \in \mathbb{R}\right\}$. Thanks to this geometrical
characterization, we shall be able to describe the structure of the
gaps of the operator ${\cal A}_s$. \\

\noindent Let us introduce the $\pi$-periodic functions $f^\pm : \mathbb{\R}
\rightarrow \mathbb{\R}^\pm$, such that, for any $\omega \in [0, \pi)$,
$$
f^+(\omega) = \tan \frac{\omega}{2} \quad \mbox{and} \quad  f^-(\omega) = -\left( \tan \frac{\omega}{2}\right)^{-1}.
$$
\noindent The easy  proof of the following result is left to the reader (see
also Figures~\ref{Explication1}, ~\ref{Explication2} and \ref{Particular cases}):
\begin{prop}
\label{gaps_ends}
An interval $(\omega_b^2 ,\omega_t^2)$ is a gap of the operator ${\cal A}_s$
if and only if $[\omega_b, \omega_t] \cap \frac{2 \pi \Z}{L} = \emptyset$ and 
one of the following three possibilities holds:
\begin{itemize} 
\item[(i)] There exists $n\in \Z$ such that $\pi n <  \omega_b <  \omega_t
  < \pi (n+1)$, and, $\phi_L(\omega_b) = f^+(\omega_b)$,
  $\phi_L(\omega_t) = f^-(\omega_t)$. 
\item[(ii)]  There exists $n\in \Z$ such that $\pi n =  \omega_b <  \omega_t
  < \pi (n+1)$, and $\phi_L(\omega_b) \leq 0$,
  $\phi_L(\omega_t) = f^-(\omega_t)$. 
\item[(iii)]  There exists $n\in \Z$ such that $\pi n <  \omega_b <  \omega_t
  = \pi (n+1)$, and $\phi_L(\omega_b) =f^+(\omega_b)$,
  $\phi_L(\omega_t) \geq 0$. 
\end{itemize}   

\end{prop}

\subsection{The discrete spectrum of ${\cal A}_s^{\mu}$}\label{SubsectionSprectreDiscretAsmu}

We are now interested in determining the discrete spectrum of
${\cal A}_s^{\mu}$. Suppose that $\omega^2$ is not in the essential spectrum
of ${\cal A}_s^{\mu}$, which implies in particular that $\omega \notin \pi
\Z$ (see Prop.~\ref{result_sym}). Let $u$ be a corresponding eigenfunction and let $\textbf{u}_j =
u(M_j^-) = u(M_j^+)$ (we consider symmetric functions). Since  
the eigenfunction $u$ verifies the equation $-u''+\omega^2u=0$ on each
horizontal edge of the graph $\mathcal{G}$, one has 
\begin{align}
&
u_{j+\frac{1}{2}}(s)=\textbf{u}_j\frac{\sin{(\omega(1-s))}}{\sin{\omega}}+\textbf{u}_{j+1}\frac{\sin{(\omega
    s)}}{\sin{\omega}},  &s:= x-j \in[0,1],\qquad & \forall
j\in\mathbb{Z}.\label{uj_def}
\end{align}
We
first begin by excluding some particular cases: 
\begin{lemme}
If $\cos \frac{\omega L}{2} = 0$, then $\omega^2$ is not in the
discrete spectrum of ${\cal A}_s^\mu$.
\end{lemme}  
\begin{proof}
If $\cos \frac{\omega L}{2} = 0$ and $\omega \notin \pi \Z$, then $\omega^2$ is an eigenvalue of infinite multiplicity (see Remark~\ref{rem:Qc}). Thus, it does not belong to the discrete spectrum of ${\cal A}_s^\mu$.   Similarly,  if $\cos \frac{\omega L}{2} = 0$ and $\omega \in \pi \Z$, then $\omega^2 \in \sigma_{ess}({\cal A}_s^\mu)$ (Prop.~\ref{result_sym}), which implies that it does not belong to the discrete spectrum of ${\cal A}_s^\mu$. 
\end{proof} 
\noindent Thus we can assume that $\cos \frac{\omega L}{2} \neq 0$.  In
this case, on the
vertical edges $e_j$, $u_j$ is given by 
\begin{align}
& u_{j}(y)=\textbf{u}_j\frac{\cos{(\omega y)}}{\cos{(\omega L/2)}},&y\in[-L/2,L/2],\qquad & \forall j\in\mathbb{Z}\label{uj1_def}.
\end{align} 
According to \eqref{uj_def}-\eqref{uj1_def}, the function $u$ is completely
determined by the point values $\textbf{u}_j$. Moreover, in order to
ensure that $u \in L^2(\mathcal{G})$, the sequence $\mathbf{u}_j$ must
be square integrable: 
\begin{equation}\label{conditionl2}
\sum_{j \in \Z } |\mathbf{u}_j|^2 < + \infty.
\end{equation}

\noindent It remains to express that $u$
belongs to $D({\cal A}_s^\mu)$ (see \eqref{Kirchoff_def}), i.e.  Kirchhoff's conditions
are satisfied. 
Doing so, we obtain the following set of finite difference equations:
\begin{align}
&\textbf{u}_{j+1}+2 \, g(\omega) \, \textbf{u}_j+\textbf{u}_{j-1}=0,\qquad j\in\mathbb{Z}^*,\label{finitediff}\\[1ex]
&\textbf{u}_{1}+2 \, g_{\mu}(\omega) \, \textbf{u}_0+\textbf{u}_{-1}=0,\label{finitediffzero}
\end{align}
with
\begin{align}
&g(\omega)=-\cos{\omega}+ \frac{\sin{\omega}}{\phi_L(\omega)}, \label{g}\\
&g^{\mu}(\omega)=-\cos{\omega}+\mu\frac{\sin{\omega}}{\phi_L(\omega)} \label{gmu}.
\end{align}
where $\phi_L$ is defined in \eqref{def_phi_f}.
Thus, we reduced the initial problem for a differential operator on the graph to a problem for a finite difference operator acting on sequences $\{\textbf{u}_j\}_{j\in\mathbb{Z}}$. Looking for particular solutions of \eqref{finitediff} for $j<0$ and
$j>0$ under the form $\mathbf{u_j} = r^j$ leads to 
the characteristic equation
\begin{equation}
\label{equation_quadratique}
r^2+2 \, g(\omega) \, r+1=0.
\end{equation}
At this point, we observe the following property
\begin{lemme} \label{LemmeSchnoll}As soon as $\cos( \omega L /2) \neq 0$, one has the
  equivalence
$$
\omega^2 \in \sigma_{ess}({\cal A}_s) \quad \Leftrightarrow \quad | g(\omega) | \leq 1.
$$
\end{lemme}
\begin{proof}
Indeed,   $|g(\omega)| \leq 1$ is equivalent to the existence of  $\theta \in [0, \pi]$
such that 
$$
\cos(\theta) = g(\omega) = -\cos{\omega}+\frac{1}{2}\sin{\omega}\tan{({\omega L}/{2})}.
$$
Since $\cos( \omega L /2) \neq 0$, this is equivalent to the
characterization~(\ref{caract_spectre_ess_sym}) of the essential spectrum.
\end{proof}
\noindent Since $\omega^2$ does not belong to the essential spectrum of
 ${\cal A}_s^\mu$, $|g(\omega)|>1$ and the discriminant $D(\omega)$ of (\ref{equation_quadratique})
is strictly positive, which means that (\ref{equation_quadratique})
has two distinct real solutions. Since the product of
these solutions is equal to one,  (\ref{equation_quadratique}) has a
unique solution $r(\omega) \in (-1,1)$ given by 
\begin{equation}
\label{r}
r(\omega)=-g(\omega)+sign(g(\omega))\sqrt{g^{2}(\omega)-1}.
\end{equation}
Joining \eqref{conditionl2}
and \eqref{finitediff}, we deduce that there exists a constant $A \neq
0$ 
\begin{equation}
\label{eigenfun_sym}
\textbf{u}_j=A \; r(\omega)^{|j|},\qquad j\in\mathbb{Z}.
\end{equation}   
It remains to enforce the Kirchhoff condition~\eqref{finitediffzero},
which leads to 
\begin{equation}
\label{rel_r_gmu}
r(\omega)=-g^{\mu}(\omega).
\end{equation}
Taking into account (\ref{g}), (\ref{gmu}) and (\ref{r}), we arrive at the following relation:
\begin{equation*}
sign(g(\omega))\sqrt{g^2(\omega)-1}=(1-\mu)(g(\omega)+\cos{\omega}).
\end{equation*}
Since $|g(\omega)|>1$, $sign(g(\omega)) = sign(g(\omega) +
\cos{\omega}) $, we can rewrite the previous equality  as
\begin{equation}
\label{mu}
F(\omega) = \mu  \quad \mbox{where} \; F(\omega):=1-\sqrt{ \frac{g^2(\omega)-1}{(g(\omega)+\cos{\omega})^2}}.
\end{equation}
For the rest of the analysis,  it is useful to rewrite $F(\omega)$ (using~\eqref{g}) as 
\begin{equation*}
F(\omega)=1-\sqrt{1-\phi_L(\omega)\left(\phi_L(\omega)+\phi_2(\omega)\right)},
\end{equation*}
where $\phi_2(\omega) = 2/\tan \omega$. 

\begin{rem} \label{RemContinuiteF} Let $(\omega_b^2, \omega_t^2)$ be a gap of the operator $\mathcal{A}_s^\mu$.  Since $|g(\omega)|>1$ in $(\omega_b, \omega_t)$, $F$ is well defined and continuous in $(\omega_b, \omega_t)$.  However, $F$ might blow up (together with the function $\phi_2$) as $\omega$ tends to $\omega_t$ or $\omega_b$, i.e. at the extremities of the gap. 
\end{rem}
\begin{thm}\label{existence_sym}
For $\mu\geq 1$, the discrete spectrum of the operator $\mathcal{A}_s^{\mu}$ is empty. For $0<\mu<1$, let $(\omega_b^2, \omega_t^2)$ be a gap of the operator $\mathcal{A}_s^\mu$: 
\begin{enumerate}
\item[(a)] If $(\omega_b^2, \omega_t^2)$ is a gap of type (i), then $\mathcal{A}_s^\mu$  has exactly two simple eigenvalues $\lambda_1 = \omega_1^2$ and $\lambda_2= \omega_2^2$ that satisfy
$\omega_b < \omega_1 < \omega_2< \omega_t$.
\item[(b)] If $(\omega_b^2, \omega_t^2)$ is a gap of type (ii) or (iii), then $\mathcal{A}_s^\mu$  has  exactly one simple eigenvalue $\lambda_1 = \omega_1^2$ such that $\omega_b < \omega_1 < \omega_t$.  
\end{enumerate}   
 \end{thm}
 
%

\begin{proof}



\noindent Assume that $\mu \geqslant 1$. If $\omega^2$ belongs to the discrete sprectrum of $\mathcal{A}_s^{\mu}$, then $\omega^2$ is in a gap of  ${\cal A}_s^\mu$ and Equation~\eqref{mu} is satisfied. But this is impossible because $|g(\omega)|>1$, which means in particular that $\mu=F(\omega)<1$.  \\

\noindent Then, we consider the case $0<\mu<1$. We investigate the variations of $F$ for the different types of gaps described in Prop.~\ref{gaps_ends}:
\begin{enumerate}
\item[-] Gap of type (i): as a preliminary step,  one can verify that  $|g(\omega_b)| = |g(\omega_t)|=1$ (using for instance the definition~\eqref{g} of $g$ together with the fact that $\phi_L(\omega_b) =f^+(\omega_b)$ and $\phi_L(\omega_t) =f^-(\omega_t)$, see Prop.~\ref{gaps_ends}), which implies that \begin{equation} \label{FomegabFomegat}
F(\omega_b)= F(\omega_t) =1.
\end{equation}
Then, let us investigate the variations of the function $\phi_L \varphi$, with $\varphi = \phi_L + \phi_2$: first, since $\phi_L(\omega_b) = f^+(\omega_b)>0$ and $\phi_L(\omega_t)=f^-(\omega_t) <0$ (see Prop.~\ref{gaps_ends}), the strictly decaying function $\phi_L$, which is continuous in the interval $[\omega_b, \omega_t]$, has exactly one zero in $(\omega_b, \omega_t)$. We denote it by $c$.  Besides, the fonction $\varphi$ is continuous and strictly decaying in the interval $[\omega_b, \omega_t]$ (Prop.~\ref{gaps_ends} ensuring the existence of $n\in \Z$ such that $[\omega_b, \omega_t] \subset (n \pi, (n+1)\pi)$, we deduce that $\phi_2$ is continuous in $[\omega_b, \omega_t]$). Moreover, it satisfies $\varphi(\omega_b)>0$ and $\varphi(\omega_t)<0$. Indeed, a direct computation shows that 
\begin{equation}\label{EncadrementPhi2}
\forall \, n\in \Z,  \quad \forall \, \omega \in (n \pi, (n+1)\pi),  \quad -f^+(\omega) < \phi_2(\omega) < -f^-(\omega).
\end{equation}
 As a consequence, $\varphi(\omega_b) = f^+(\omega_b) + \phi_2(\omega_b)>0$ and $\varphi(\omega_t) = f^-(\omega_t) + \phi_2(\omega_t)<0$.  As a result, $\varphi$ has exactly one zero in $(\omega_b, \omega_t)$. We denote it by $d$.

Noting that~\eqref{FomegabFomegat} implies that $\phi_L(\omega_b) \varphi(\omega_b) =  \phi_L(\omega_t) \varphi(\omega_t)=1$, we  deduce that the function $\phi_L \varphi$, which is continuous on $[\omega_b, \omega_t]$, is strictly decaying from $1$ to $0$ in the interval $[\omega_b, \min(c,d)]$, is strictly increasing from $0$ to $1$  in the interval $[\max(c,d), \omega_t]$, and is negative in the interval $(\min(c,d), \max(c,d))$.
It follows that $F$, which is therefore also continuous in $[\omega_b, \omega_t]$, is strictly decaying from $1$ to $0$ in the interval $[\omega_b, \min(c,d)]$, is negative in the interval $(\min(c,d), \max(c,d))$, and is strictly increasing from $0$ to $1$ in $[\max(c,d), \omega_t]$. As a result, for any $\mu\in(0,1)$, Equation~\eqref{mu} has exactly two solutions in $(\omega_b, \omega_t)$, the first one belonging to  $(\omega_b, \min(c,d))$ and the second one to $(\max(c,d), \omega_t)$.
\item[-] Gap of type (ii): in this case, $\omega_b \in \Z \pi$ and the function $F$ blows up in the neighborhood $\omega_b$ unless $\cos( \omega_b L/2)=0$.  More precisely, we can prove that 
$$
\lim_{\omega\rightarrow \omega_b^+} F(\omega)= 
  \begin{cases} 
1 - \sqrt{1 + 2 L}   < 0 &  \mbox{if} \; \cos( \omega_b L/2) = 0, \\
- \infty & \mbox{otherwise}.
\end{cases} 
$$ 
By contrast, since $\omega_t \notin \Z \pi$ and as for the first kind of gap,  we can prove that
\begin{equation}\label{GapType2top}
F(\omega_t) =1.
\end{equation}  
Then, here again, we investigate the variations of the function $\phi_L \varphi$, with $\varphi = \phi_L + \phi_2$. In view of Prop~\ref{gaps_ends}, the function $\phi_L$ is continuous, strictly decaying and negative in the intervall $(\omega_b, \omega_t]$.
Then,  the function $\varphi= \phi_L + \phi_2$ is continuous in $(\omega_b, \omega_t]$, strictly decaying, and (thanks to~\eqref{EncadrementPhi2}) satisfies
$$
\lim_{\omega\rightarrow \omega_b^+} \varphi(\omega) = + \infty \quad \mbox{and} \lim_{\omega\rightarrow \omega_t^-} \varphi(\omega)  = f^-(\omega_t) + \phi_2(\omega_t)< 0. 
$$

 As result, $\varphi$ has still exactly one zero in $(\omega_b, \omega_t)$. We denote it by $d$.


Noting that~\eqref{GapType2top} implies that $\phi_L(\omega_t) \varphi(\omega_t)=1$, we deduce that the function $\phi_L \varphi$, which is continuous in $(\omega_b, \omega_t]$, is negative in $(\omega_b, d)$, and stricly increasing from $0$ to $1$ in $[d, \omega_t]$. Thus, the function $F$, which is continuous in $(\omega_b, \omega_t)$, is negative in $(\omega_b, d)$, and strictly increasing from $0$ to $1$ on $[d, \omega_t]$. Consequently, for any $\mu \in (0,1)$,  Equation~\eqref{mu} has exactly one solution (that belongs to $(d, \omega_t)$). The proof for the gaps of type (iii) follows the same way. 
\end{enumerate}

\end{proof}
\subsection{The spectrum of the operator ${\cal A}_{a}^{\mu}$.} \label{SectionSpectreAntisymetrique}
We will now briefly describe the modifications of the previous considerations in the case of the operator ${\cal A}_{a}^{\mu}$.
The operator corresponding to the periodic case $\mu=1$ is denoted by ${\cal A}_{a}$. First, based on compact perturbation arguments,  we can prove the proposition, which is analogous to Proposition ~\ref{proposition_spectre}:
\begin{prop}
\label{proposition_spectre_a} The essential spectra of ${\cal A}_a^{\mu}$ and
${\cal A}_a$ coincide:
\begin{equation}
 \sigma_{ess}({\cal A}_a^{\mu})=\sigma_{ess}({\cal A}_a).
\end{equation}
\end{prop}
\noindent Besides, using the Floquet-Bloch Theory, we obtain the analogue of Proposition \ref{eigenvalues_sym} in the antisymmetric case (we refer the reader to Section~\ref{SectionGrapheSym} for the definition of ${\cal A}_{a}(\theta)$):
\begin{prop}
\label{eigenvalues_asym}
For $\theta\in[0,\pi]$, $\omega^2\in\sigma({\cal A}_{a}(\theta))$ if and only if $\omega\neq0$  and $\omega$ is a solution of the equation
\begin{equation}
\label{caract_spectre_ess_asym}
2\sin ({\omega L}/{2}) \left(\cos{\omega}-\cos{\theta}\right)=-\sin{\omega}\cos({\omega L}/{2}).\end{equation}
\end{prop}
\noindent Thanks to the previous characterization, and similarly to the results of Proposition \ref{result_sym},  we can describe the structure of the spectrum of ${\cal A}_a$:
\begin{prop}\label{result_asym} The following properties hold:
\begin{enumerate}
\item $\sigma_{2,a} \cup \sigma_{L,a} \subset \sigma({\cal A}_s)$, where  
$\sigma_{2,a} = \left\{(\pi n)^2,\; n\in\mathbb{N}^\ast\right\}$ and $\sigma_{L,a} = \left\{  \big( (2n+1) \pi /L\big)^2,\; n\in\mathbb{N}\right\}$.

\item The operator ${\cal A}_a$ has infinitely many gaps whose ends tend to infinity.
\end{enumerate}

\end{prop}

\noindent Then, excluding here again the particular cases $\omega \in \{\pi\mathbb{Z}\}\cup\{2\pi\mathbb{Z}/L\}$, the computation of the discrete spectrum ${\cal A}_{a}$   leads to the set (\ref{finitediff})-(\ref{finitediffzero}) of finite-difference equations  substituting  $g_{a}(\omega)$ and $g_{a}^{\mu}(\lambda)$ for respectively $g(\omega)$ and $g^{\mu}(\omega)$: 
\begin{align*}
&g_{a}(\lambda)=-\cos{\lambda}+\frac{1}{2}\sin{\lambda}\tan{(\lambda L/2+\pi/2)},\\
&g_{a}^{\mu}(\lambda)=-\cos{\lambda}+\frac{\mu}{2}\sin{\lambda}\tan{(\lambda L/2+\pi/2)}.
\end{align*}
The investigation of the characteristic equation (\ref{equation_quadratique}) then provides the following characterization for the discrete spectrum of ${\cal A}_{a}$: 
 \begin{equation}\label{CharacValeurPropreA}
\omega^2\in\sigma_d({\cal A}_{a}^{\mu})\quad \Leftrightarrow\quad \mu = 1-\sqrt{ \frac{g_{a}^2(\omega)-1}{(g_{a}(\omega)+\cos{\omega})^2}}.
\end{equation}
Finally, as in the symmetric case (see Theorem ~\ref{existence_sym}), a detailed analysis of~\eqref{CharacValeurPropreA} allows us to prove the following result of the existence of eigenvalues: 
\begin{thm}
For $\mu\geqslant 1$ the discrete sprectrum of ${\cal A}_{a}^{\mu}$ is empty. For $0<\mu<1$,  there exists either one or two eigenvalues in each gap of ${\cal A}_{a}^{\mu}$. \end{thm}
\subsection{The spectrum of the operator ${\cal A}$.} As we have seen, both of the operators ${\cal A}_s$ and ${\cal A}_{a}$ have infinitely many gaps. However, it turns out that the gaps of one operator overlap with the spectral bands of the other one, so that the full operator ${\cal A}$ have no gap.
\begin{prop}
\begin{equation*}
\sigma({\cal A})=\mathbb{R}^+.
\end{equation*}
\end{prop}
\begin{proof}
Let us suppose that there exists $\omega$ such that $\omega^2\notin\sigma({\cal A})$ (of course, the same is true for some open neighborhood of $\omega$). We first note that $\omega \notin \sigma_{L} \cup \sigma_{L,a}$, since these sets are either in the spectrum of  ${\cal A}_{s}^{\mu}$ or in the spectrum ${\cal A}_{a}^{\mu}$ (Propositions ~\ref{result_sym} and \ref{result_asym}). As a consequence,  $ \cos(\omega L/2)\neq 0$ and $\sin(\omega L/2)\neq0$. Then, since $\omega^2\notin\sigma({\cal A})$,  the characterizations  \eqref{caract_spectre_ess_sym}-\eqref{caract_spectre_ess_asym}  of the essential spectrum of ${\cal A}_{s}$ and ${\cal A}_{a}$ (divided respectively by $\cos(\omega L/2)$ and $ \sin(\omega L/2)$) imply that
\begin{equation}
\label{systemR}
\left|-\cos{\omega}+\displaystyle\frac{1}{2}\sin{\omega}\tan{(\omega L/2)}\right|>1\quad \mbox{and} \quad 
\left|\cos{\omega}+\displaystyle\frac{\sin{\omega}}{2\tan{(\omega L/2)}}\right|>1.
\end{equation}
Introducing $a=\tan(\omega L/2)$, the system (\ref{systemR}) can be rewritten as
\begin{numcases}
\strut \frac{a^2}{4}\sin^2{\omega}-a\sin{\omega}\cos{\omega}+\cos^2{\omega}>1,\label{systR1}\\
\frac{1}{4a^2}\sin^2{\omega}+\frac{1}{a}\sin{\omega}\cos{\omega}+\cos^2{\omega}>1.\label{systR2}
\end{numcases}
Multiplying (\ref{systR2}) by $a^2$ and taking the sum with (\ref{systR1}) we obtain
\begin{equation*}
\frac{1}{4}(1+a^2)\sin^2{\omega}+(1+a^2)\cos^2{\omega}>1+a^2,
\end{equation*}
which is impossible.
\end{proof}
\noindent Let us then remark that the set of eigenvalues of ${\cal A}^\mu$, which is the union of the sets of eigenvalues of ${\cal A}^\mu_s$ and ${\cal A}^\mu_a$, is embedded in the essential spectrum of ${\cal A}^\mu$.
\section{Existence of eigenvalues for the operator on the ladder}\label{sec:devasymp_ladder}
\subsection{Main result }\label{sub:methodology}
We return now to the case of the ladder. As it was mentioned before,
instead of studying the full operator $A^{\mu}_{\varepsilon}$ we will
study separately the operators $A^{\mu}_{\varepsilon,s}$,
$A^{\mu}_{\varepsilon,a}$.
Let us remind first the
 result, already proven for instance in 
\cite{KuchmentZeng:2001},
which states the convergence of the essential spectrum of the periodic
operators $A_{\varepsilon,s}$ (resp. $A_{\varepsilon,a}$)  to the essential
spectrum of  $\mathcal{A}_{s}$ (resp. $\mathcal{A}_{a}$).  The proof in~\cite{KuchmentZeng:2001} is based on the convergence of the eigenvalues of the reduced operator $A_\varepsilon(\theta)$.  We point out that the construction of the asymptotic expansion of these eigenvalues in the vicinity of the intersection point of the dispersion curves of $A_\varepsilon(\theta)$ is delicate and we refer the reader to~\cite{MR2766589} for an example of detailed asymptotic in that case}.
\begin{thm}[Essential spectrum]
\label{thmKuchment}
Let $\left\{(a_m, b_m),\:m\in\mathbb{N}^*\right\}$ be the gaps of the operator $\mathcal{A}_s$ (respectively $\mathcal{A}_{a}$) on the limit graph $\mathcal{G}$. Then, for each $m_0\in\mathbb{N}^*$ there exists $\varepsilon_0 >0$ such that if $\varepsilon<\varepsilon_0$ the operator $A_{\varepsilon,s}$ (respectively $A_{\varepsilon,a}$) has at least $m_0$ gaps $\left\{(a_{\varepsilon,m}, b_{\varepsilon,m}),\:1\leqslant m\leqslant m_0\right\}$ such that
\begin{equation*}
a_{\varepsilon,m}=a_m+O(\varepsilon),\quad b_{\varepsilon,m}=b_m+O(\varepsilon),\qquad\varepsilon\rightarrow 0,\qquad 1\leqslant m\leqslant m_0. 
\end{equation*}
\end{thm} 
\noindent  In \cite{Post:2006}, O. Post proves the norm convergence of the resolvent
  of the laplacian with Neumann boundary conditions for a
   large class of thin domains shrinking to graphs. It
  consequently demonstrates the existence of eigenvalues of
  $A_{\varepsilon,s}^\mu$  (respectively $A_{\varepsilon,a}^\mu$)  located in the gap of the essential
  spectrum. This paper provides a simple and constructive alternative proof of this result. In the following proof, we consider the eigenvalues of the operator $A^{\mu}_{\varepsilon,s}$, the case of the operator $A^{\mu}_{\varepsilon,a}$ being treated analogously. 
   \begin{thm}[Discrete spectrum]
\label{existence_eig}
Let $(a,b)$ be a gap of the operator $\mathcal{A}^{\mu}_s$ (respectively $\mathcal{A}^{\mu}_{a}$) on the limit graph $\mathcal{G}$ and $\lambda\in\;(a,b)$ an eigenvalue of this operator. Then there exists $\varepsilon_0 >0$ such that if $\varepsilon<\varepsilon_0$ the operator $A^{\mu}_{\varepsilon,s}$ (resp. $A^{\mu}_{\varepsilon,a}$) has an eigenvalue $\lambda_{\varepsilon}$ inside a gap $(a_{\varepsilon}, b_{\varepsilon})$. Moreover, for $\varepsilon<\varepsilon_0$, there exists $C>0$ such that
\begin{equation}\label{ConvergenceValeurPropre}
| \lambda_{\varepsilon} - \lambda | \leq C \sqrt{\varepsilon}. 
\end{equation}
\end{thm} 
 
\begin{rem}\label{rem:conv}
As every eigenvalue of the operators $\mathcal{A}^{\mu}_s$ (resp. $\mathcal{A}^{\mu}_{a}$) is simple (as established in Theorem \ref{existence_sym}), for $\varepsilon$ small enough, $\lambda_\varepsilon$ will be a simple eigenvalue of $A^{\mu}_{\varepsilon,s}$ (resp. $A^{\mu}_{\varepsilon,a}$), see \cite{Post:2006}. In addition, it is worth noting that  the error estimate~\eqref{ConvergenceValeurPropre} is suboptimal. In fact,  writing a high order asymptotic expansion of $\lambda^\varepsilon$ restores the optimal convergence rate:
$$
| \lambda_{\varepsilon} - \lambda^{(0)} | \leq C {\varepsilon} 
$$ 
The construction and the justification of this high order asymptotic expansion will be detailed in a forthcoming paper.
\end{rem} 

\begin{rem} \label{remarkDirichletLadder}We point out that imposing Dirichlet conditions leads to an entirely different asymptotic analysis. In~\cite{Grieser2008}, the asymptotic of the eigenvalues is obtained in the case of a compact 'thickened' graph with different types of boundary conditions, including the Dirichlet and Robin ones (see also \cite{panasenko2006} for non standard boundary conditions). Besides,  the Dirichlet ladder is investigated in \cite{NazarovLadder2012}-\cite{NazarovLadder2014}: as in our case,  changing the size of one or several rungs of the ladder can create eigenvalues inside the first gap ( see~\cite[Theorem 8.1]{NazarovLadder2014}). The analysis is deeply linked to the presence of a non empty discrete spectrum for the Laplace problem posed in a T-shape waveguide (cf.~\cite{Nazarov2010Tshape}).    
\end{rem}
Thanks to Theorem \ref{existence_eig}, it is easy to show the existence of as many eigenvalues as one wants.
\begin{corollaire}
	For any number $m\in \N$, it exists $\eps_0$ such that for all $\eps\leq \eps_0$, $A_{\varepsilon,s}^\mu$  (respectively $A_{\varepsilon,a}^\mu$) has at least $m$ eigenvalues.
	\end{corollaire}
	In the next section, we give the constructive proof and in Section \ref{sub:numerical_illustration}, we illustrate these theoretical results by numerical illustrations.
\noindent 
\subsection{Proof of Theorem~\ref{existence_eig}}
Our proof of Theorem \ref{existence_eig} relies on the construction of a pseudo-mode,
that is to say a symmetric function $u_{\varepsilon}\in H^1(\Omega_{\varepsilon}^{\mu})$ such that for every symmetric function $v\in H^1(\Omega_{\varepsilon}^{\mu})$
\begin{equation}
\label{est_pseudomode}
 \Big|\int\limits_{\Omega_{\varepsilon}^{\mu}}\left(\nabla u_{\varepsilon}\nabla v-\lambda u_{\varepsilon}v\right)\,d{\bf x}\Big|\leqslant C \, \sqrt{\varepsilon} \,\|u_{\varepsilon}\|_{H^1(\Omega_{\varepsilon}^{\mu})}\|v\|_{H^1(\Omega_{\varepsilon}^{\mu})},
 \end{equation}
By adapting the Lemma 4 for \cite{MR2766589} (see Appendix A in \cite{RR}) the existence of such a function provides an estimate of the
 distance from $\lambda$ to the spectrum of $A^{\mu}_{\varepsilon,s}$, namely
 \begin{equation}
 \label{dist_spectrum}
 dist(\sigma(A^{\mu}_{\varepsilon,s}),\lambda)\leqslant \widetilde{C} \sqrt{\varepsilon},
 \end{equation}
 with some constant $\,\widetilde{C}\,$ that does not depend on
 $\,\varepsilon$, but depends on $\lambda$.\\\\
 According to Theorem \ref{thmKuchment}, for $\varepsilon$
 small enough, there exists a constant $C$ such that
 $\,\sigma_{ess}(A^{\mu}_{\varepsilon,s})\cap
 [a+C\varepsilon,\,b-C\varepsilon]=\emptyset\,$. As a consequence, the
 intersection between the discrete spectrum
 $\sigma_{d}(A^{\mu}_{\varepsilon,s})$ and the interval $ [\lambda-~\widetilde{C}\sqrt{\varepsilon},\,\lambda+\widetilde{C}\sqrt{\varepsilon}]$
 is non empty, which proves the existence of an eigenvalue in the neighborhood of $\lambda$.  
 
 \subsubsection{Construction of a pseudo-mode}
Since we consider the symmetric case, it suffices to construct the pseudo-mode $u_{\eps}$ 
on the lower half part ${\Omega}_{\varepsilon}^{\mu,-} $
of $\Omega_{\varepsilon}^{\mu}$ (comb shape domain, see Figure~\ref{bande_notations}):
$$
{\Omega}_{\varepsilon}^{\mu,-} = \left\{ (x,y) \in \Omega_{\varepsilon}^{\mu} \; \mbox{ \;s.t. } \;
  y < 0\right\}.
$$
\begin{figure}[htbp]
\begin{center}
\includegraphics[scale=0.4, trim=0cm 8.5cm 0cm 8cm, clip]{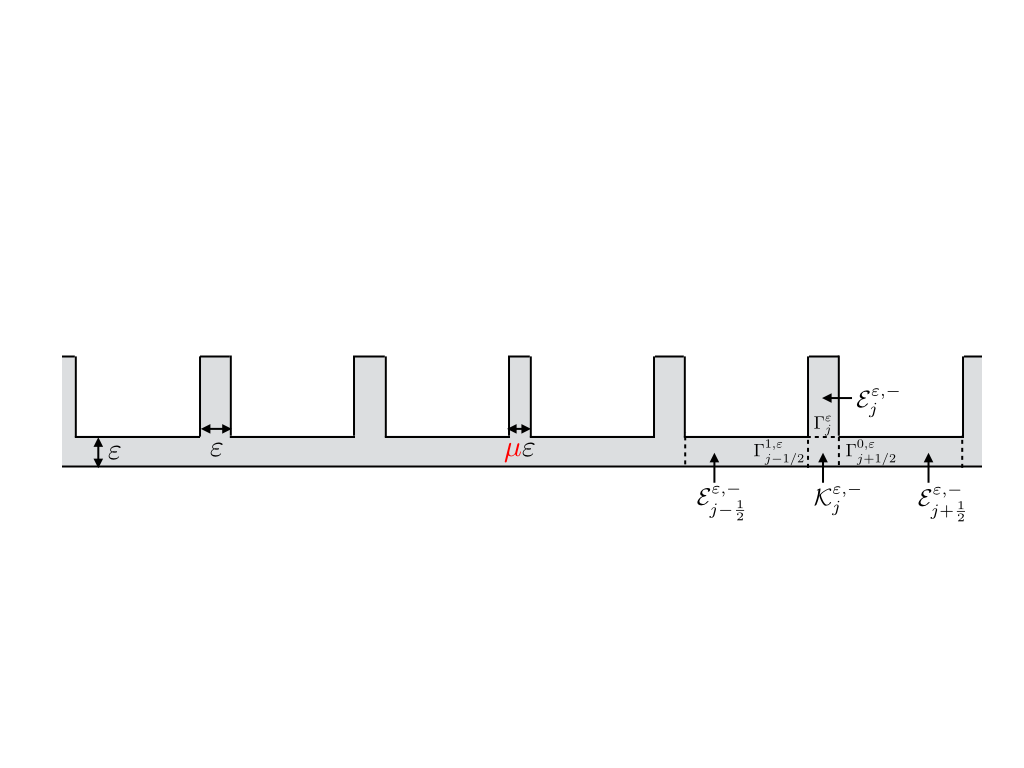}
\caption{\label{bande_notations}The domain ${\Omega}_{\varepsilon}^{\mu,-} $}
\end{center}
\end{figure}

\noindent As represented on Figure~\ref{bande_notations}, we denote by $\mathcal{E}_{j+\frac{1}{2}
 }^{  \varepsilon,-} $, $j\in\mathbb{Z}$, the horizontal
edges of the domain $\Omega_{\varepsilon}^{\mu,-}$
\begin{align*}
 \mathcal{E}_{j+\frac{1}{2}
}^{  \varepsilon,-} = \left(j+\varepsilon \mu_j/2, (j+1) -
\varepsilon \mu_{j+1}/2\right)\times \left( -L/2, -L/2+
\varepsilon\right),
\end{align*}
by $\mathcal{E}_{j}^{  \varepsilon,-}$, $j\in\mathbb{Z}$, its
vertical edges
\begin{equation*}
\mathcal{E}_{j
}^{ \varepsilon,-}  = \left(j- \varepsilon \mu_j/2, j+ \varepsilon
\mu_j/2\right) \times \left( -L/2 + \varepsilon,
0\right),
\end{equation*}
and by $\mathcal{K}_j^{\varepsilon,-}$, $j\in\mathbb{Z}$, the junctions
\begin{align*}
& \mathcal{K}_j^{\varepsilon,-} = \left(j- \varepsilon \mu_j/2, j+ \varepsilon
\mu_j/2\right) \times \left( -L/2,
-L/2 + \varepsilon\right). 
\end{align*}
Here,  $\mu_j=1$ if $j\neq 0$ and $\mu_0=\mu$ (with the notations of Section \ref{SectionDefinitionNotationAmu} $\mu_j=w^\mu(e_j)$, the function $w^\mu$ being defined by \eqref{eq:w_mu}).\\

\noindent Denoting by $u$
an eigenfunction  of the limite operator $\mathcal{A}^{\mu}_s$ associated with the
eigenvalue $\lambda$ { (see formula~\eqref{uj1_def})},  we construct the pseudo-mode $u_{\varepsilon}$ on $\Omega_{\varepsilon}^{\mu,-}$ by "fattening" $u$ (with an appropriate rescaling) as follows:
\begin{equation*}
u_{\varepsilon}(x,y)=\left\{
\begin{array}{lll}
u_{j+\frac{1}{2}}(s_{j+1/2}^\varepsilon(x)),\quad&(x,y)\in\mathcal{E}_{j+\frac{1}{2}}^{\varepsilon,-},\\
u_{j}(t^\varepsilon(y)),&(x,y)\in\mathcal{E}_{j}^{\varepsilon,-},\\
\textbf{u}_{j},&(x,y)\in \mathcal{K}_{j}^{\varepsilon,-}.
\end{array}
\right.
\end{equation*}
Here the rescaling functions $s_{j+1/2}^\varepsilon$ and
$t^\varepsilon$ are linear functions given by the relations  
\begin{equation}
\label{change_var}
s_{j+1/2}^\varepsilon(x)=\frac{x-j-w^{\mu}(e_j)\varepsilon/2}{1-\left(w^{\mu}(j)+w^{\mu}(e_{j+1})\right)\varepsilon/2},
\qquad t^\varepsilon(y)=\frac{y}{1-2\varepsilon/L}.
\end{equation}
{\noindent We remark that the function $u_\varepsilon \in H^1(\Omega_\varepsilon^{\mu, -})$ since the function $u$ is continuous at the vertices of the graph.}
\subsubsection{Proof of Estimate~(\ref{est_pseudomode})}
The pseudo-mode being constructed, it remains to prove~\eqref{est_pseudomode}. We notice that it is sufficient to prove it for any test function
$v\in \mathcal{C}^1_s(\overline{\Omega_{\varepsilon}^{\mu}})$ ($\mathcal{C}^1_s$ standing
for the symmetric subspace of $\mathcal{C}^1$). Indeed, 
$\mathcal{C}^1_s(\overline{\Omega_{\varepsilon}^{\mu}})$ is dense in the subset of $H^1({\Omega_{\varepsilon}^{\mu}})$ made of symmetric functions. Let us then estimate
the left-hand side of (\ref{est_pseudomode}) for $v\in
\mathcal{C}^1_s(\overline{\Omega_{\varepsilon}^{\mu}})$.  
{ First, an integration by parts gives
\begin{multline}
\int_{\mathcal{E}_{j-1/2}^{\varepsilon, -}} \nabla u_{\varepsilon}\nabla v-\lambda u_{\varepsilon}v\,d{\bf x}  =  \int_{\mathcal{E}_{j-1/2}^{\varepsilon, -}}  \underbrace{\left( {u_{j-1/2}}''(s^\varepsilon_{j-1/2} (x)) -  \lambda  u_{j-1/2}(s^\varepsilon_{j-1/2} (x)) \right)}_{=0} v dx \\ +  \int_{\mathcal{E}_{j-1/2}^{\varepsilon, -}}   \left( (s^\varepsilon_{j-1/2})'(x) -1\right) {u_{j-1/2}}''(s^\varepsilon_{j-1/2} (x))  v dx  + \int_{\Gamma_{j-1/2}^{1,\varepsilon} \cup {\Gamma_{j-1/2}^{0,\varepsilon}}} \partial_n   u_{\varepsilon}  v dx \\
=  \lambda  \int_{\mathcal{E}_{j-1/2}^{\varepsilon, -}}   \left( (s^\varepsilon_{j-1/2})'(x) -1\right) {u_{j-1/2}}(s^\varepsilon_{j-1/2} (x))  v dx  + \int_{\Gamma_{j-1/2}^{1,\varepsilon} \cup {\Gamma_{j-1/2}^{0,\varepsilon}}} \partial_n   u_{\varepsilon}  v dx,
\end{multline} 
where, for any $j \in \mathbb{Z}$,
$\,\Gamma_{j-\frac{1}{2}}^{0,\varepsilon}=\partial\mathcal{E}_{j-\frac{1}{2}}^{\varepsilon,-}\cap
\partial \mathcal{K}_{j-1}^{\varepsilon,-}$,
$\,\Gamma_{j-\frac{1}{2}}^{1,\varepsilon}=\partial\mathcal{E}_{j-\frac{1}{2}}^{\varepsilon,-}\cap
\partial \mathcal{K}_j^{\varepsilon,-}$ (see Figure~\ref{bande_notations}). Here, we used the fact that $s^\varepsilon_{j-1/2}$  is a linear function of $x$. Similarly, 
\begin{multline}
\int_{\mathcal{E}_{j}^{\varepsilon, -}} \nabla u_{\varepsilon}\nabla v-\lambda u_{\varepsilon}v\,d{\bf x}  
=  \lambda  \int_{\mathcal{E}_{j}^{\varepsilon, -}}   \left( (t^\varepsilon_{j})'(x) -1\right) {u_{j-1/2}}(t^\varepsilon_{j} (x))  v dx  + \int_{\Gamma_{j}^{\varepsilon} } \partial_n   u_{\varepsilon}  v dx \
\end{multline} 
where $\Gamma_{j}^{\varepsilon}=\partial\mathcal{E}_{j}^{\varepsilon}\cap
\mathcal{K}_j^{\varepsilon,-}$. Summing over $j \in \Z$ (noting that $|(s^\varepsilon_{j-1/2})'(x) -1 |$ and $| (t^\varepsilon)' (y)-1 |$ are of order $\varepsilon$),  we obtain}
\begin{multline}
\label{proof_pseudomode}
 \Big|\int\limits_{\Omega_{\varepsilon}^{\mu,-}}\left(\nabla u_{\varepsilon}\nabla v-\lambda u_{\varepsilon}v\right)\,d{\bf x}\Big| 
 \leqslant \sum\limits_{j\in\mathbb{Z}}\lambda \textbf{u}_{j}\Big|\int\limits_{\mathcal{K}_{j}^{\varepsilon,-}}v\,d{\bf x}\Big|+\lambda\|u_{\varepsilon}\|_{L_2(\Omega_{\varepsilon}^{\mu,-})}\|v\|_{L_2(\Omega_{\varepsilon}^{\mu,-})}O(\varepsilon)\\
+\sum\limits_{j\in\mathbb{Z}}\Big(-u'_{j+\frac{1}{2}}(0)\int\limits_{\Gamma_{j+\frac{1}{2}}^{0,\varepsilon}}v(x,y)dy+u'_{j-\frac{1}{2}}(1)\int\limits_{\Gamma_{j-\frac{1}{2}}^{1,\varepsilon}}v(x,y)dy-u'_{j}\left(-\textstyle\frac{L}{2}\right)\int\limits_{\Gamma_{j}^{\varepsilon}}v(x,y)dx\Big)\left(1+O(\varepsilon)\right),
\end{multline}
Since $\textbf{u}_{j}$ is
a geometrical progression (according
to~\eqref{eigenfun_sym}), and using the Cauchy Schwarz inequality (the size of the junction
$\mathcal{K}_j^{\varepsilon,-}$ is of 
order $\varepsilon^2$), we obtain
\begin{equation}
\label{proof_pseudomode1}
\sum\limits_{j\in\mathbb{Z}}\lambda \textbf{u}_{j}\Big|\int\limits_{\mathcal{K}_{j}^{\varepsilon,-}}v\,d{\bf x}\Big|\leqslant C_1\varepsilon\|v\|_{L_2(\Omega_{\varepsilon}^{\mu,-})}.
\end{equation}
where $C_1$ is a constant that does not depend on $\varepsilon$.
{ Next, denoting by $M^{\varepsilon,-}_j$ the barycenter of $\mathcal{K}_{j}^{\varepsilon, -}$, we remark that we can replace $v(x,y)$ by $v(x,y)-v\left(M^{\varepsilon,-}_j\right)$ in the integrals over the boundaries in the right-hand side of (\ref{proof_pseudomode}) because $u$  satisfies the Kirchhoff's
conditions~\eqref{Kirchoff_def}}. Moreover, 
\begin{equation}
\label{proof_pseudomode2}
\Big|\int\limits_{\Gamma_{j}^{\varepsilon}}\left(v(x,y)-v\left(M^{\varepsilon,-}_j\right)\right)\,dx\Big|\leqslant \int\limits_{\Gamma_{j}^{\varepsilon}}\int\limits_{M_j^{\varepsilon,-}}^{(x,y)}|\nabla v|\,dxdt\leqslant C_2\varepsilon\|v\|_{H^1(\mathcal{K}_j^{\varepsilon,-})}.
\end{equation}
{ In the previous formula, $\int\limits_{M_j^{\varepsilon,-}}^{(x,y)}$  stands for the integral on the segment linking $M_j^{\varepsilon,-}$ to the point of coordinates $(x,y)$.}
Combining (\ref{proof_pseudomode}-\ref{proof_pseudomode1}-\ref{proof_pseudomode2}) and 
taking into account (\ref{uj_def}-\ref{uj1_def}-\ref{eigenfun_sym}), we
obtain that
\begin{equation}
\label{proof_pseudomode3}
\Big|\int\limits_{\Omega_{\varepsilon}^{\mu,-}}\left(\nabla u_{\varepsilon}\nabla v-\lambda u_{\varepsilon}v\right)d{\bf x}\Big|\leqslant C_3\varepsilon\|v\|_{H^1(\Omega_{\varepsilon}^{\mu,-})},\qquad\forall v\in \mathcal{C}^1_s(\Omega_{\varepsilon}^{\mu}).
\end{equation}
To conclude, we notice that by definition of the pseudo-mode $u_{\varepsilon}$,
\begin{equation*}
\|u_{\varepsilon}\|_{H^1(\Omega_{\varepsilon}^{\mu,-})}\geqslant C_4\sqrt{\varepsilon}\|u\|_{H^1(\mathcal{G}^-)},\qquad C_4>0,
\end{equation*}
($\mathcal{G}^-$ standing for the lower half part of the graph $\mathcal{G}$),  which, together with (\ref{proof_pseudomode3}) and the density argument mentioned above finishes the proof of (\ref{est_pseudomode}).
\subsection{Numerical illustration} 
\label{sub:numerical_illustration}
To illustrate and validate the results of Theorem \ref{existence_eig}, we have computed a part of the essential spectrum, some eigenvalues and their associated eigenvectors of the operator $A_{\eps,s}^\mu$ for several values of $\eps$ and several values of $\mu$. 
\subsubsection*{Essential spectrum}
To compute the essential spectrum of the operator $A_{\eps,s}^\mu$ (resp. $\mathcal{A}_s$), one method consists in computing the eigenvalues $\lambda^{(n)}_\eps(\theta) = (\omega^{(n)}_\eps(\theta))^2$ defined in \eqref{eigs} (resp. $\lambda^{(n)}_s(\theta)=(\omega^{(n)}_s(\theta))^2$) defined in \eqref{eig}) for a discrete set of $\theta$ included in $[-\pi,\pi[$ (or equivalently in $[-\pi,0]$). {From a numerical point of view, this is done using the standard $P_1$ conform finite element method (\cite{GiraultRaviart}). }We have represented in Figure~\ref{fig:dispersive_curves} the dispersive curves $\theta\mapsto \omega^{(n)}_\eps(\theta)$ for $n\in\llbracket 1,5\rrbracket$, when $L=2$ and $\eps=0.1$ (left figure), and $\theta\mapsto\omega^{(n)}_s(\theta)$ for $n\in\llbracket 1,5\rrbracket$, corresponding to the graph with the same $L$ (right figure). 
\begin{figure}[htbp]
\begin{center}
\includegraphics[scale=0.3, trim = 0cm 0cm 0cm 1.5cm, clip]{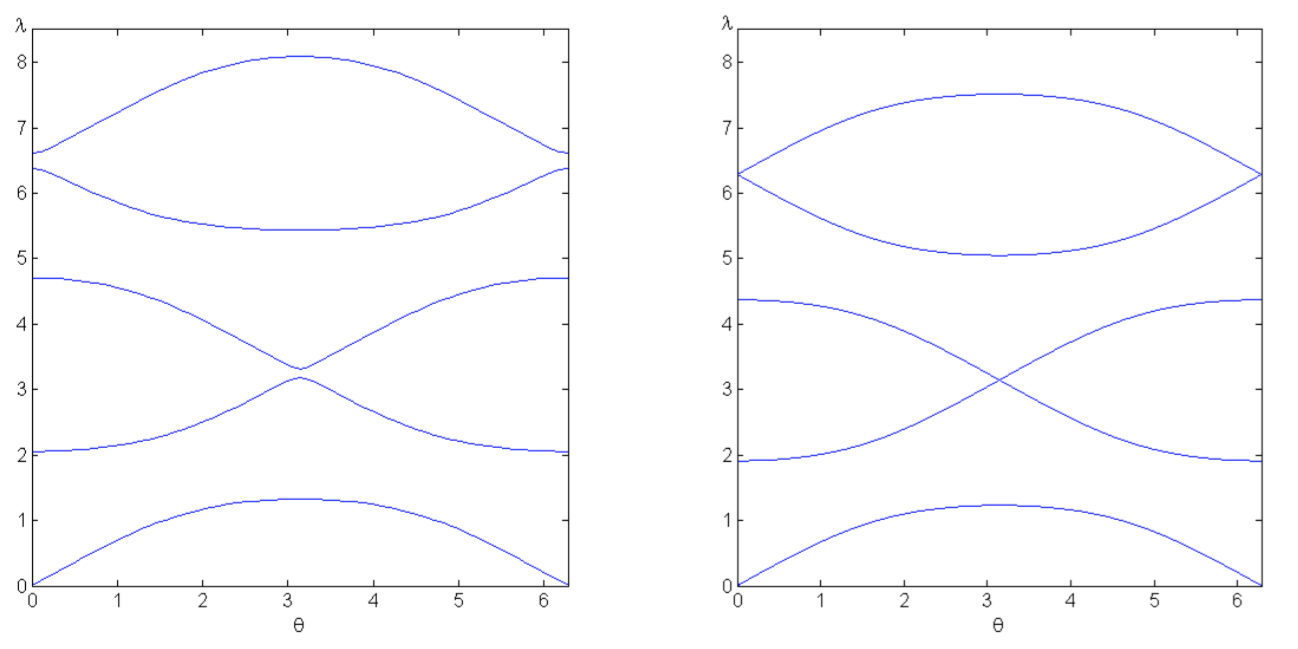}
\caption{\label{fig:dispersive_curves} Dispersive curves for the ladder  of thickness $\eps=0.1$ (left figure) and for the graph (right figure) when $L=2$.}
\end{center}
\end{figure}
The essential spectrum can be easily deduced from the dispersion curves: indeed, as explained by the Floquet-Bloch theory,  it is the image of the segment $([-\pi,\pi])$ by the functions $\lambda^{(n)}_\eps$ (resp. $\lambda^{(n)}_s$). In Figure~\ref{fig:essential_spectrum}, we have represented a part of the essential spectrum of the operator $A_{\eps,s}^\mu$ for different values of $\eps$: the blue bars correspond to the values $\omega$ such that $\lambda = \omega^2$ is in the essential spectrum of $A_{\eps,s}^\mu$.
\begin{figure}[htbp]
\begin{center}
\includegraphics[scale=0.2, trim = 0cm 0cm 0cm 1.8cm, clip]{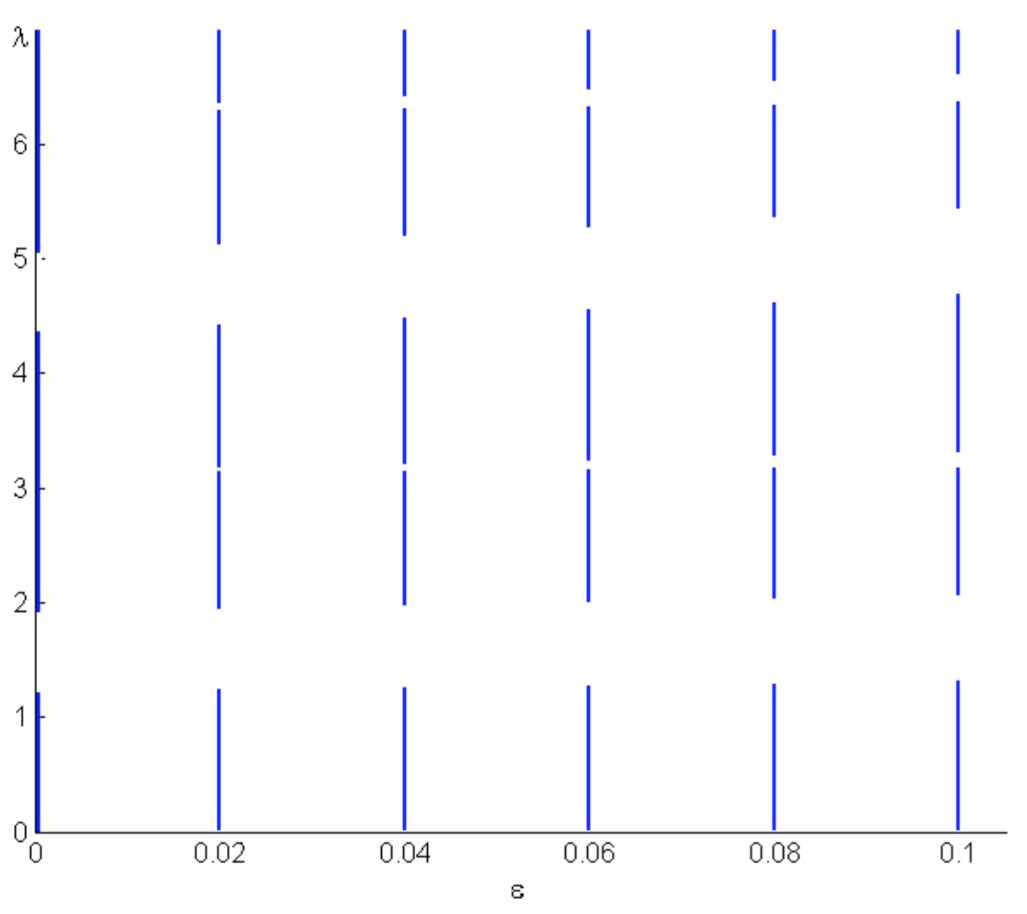}
\caption{\label{fig:essential_spectrum} Representation of the essential spectrum of $A_{\eps,s}^\mu$ for different values of $\eps$ for $L=2$}
\end{center}
\end{figure}
Obviously, for small values of $\eps$, the essential spectrum of the operator $A_{\eps,s}^\mu$ is very close to the essential spectrum of the limit operator $\mathcal{A}_s^\mu$. More precisely,  to each gap of the limit operator $\mathcal{A}_s^\mu$, corresponds a gap of the operator $A_{\eps,s}^\mu$, which is close to it for small $\eps$ . The convergence with respect to $\eps$ is linear as it is predicted by the theory, see Theorem~\ref{thmKuchment}. We also remark a phenomenon that has not been detected by our approach: the opening of a gap near the values $\omega = {\pi\N^*}$, points where the dispersion curves of the limit operator $\mathcal{A}_s^\mu$ touch (see right figure of Figure \ref{fig:dispersive_curves}). This phenomenon could be probably proven using the techniques of \cite{MR2766589}.
\\\\
Another interesting phenomenon concerns the eigenvalues of infinite multiplicity for the limit operator. As explained in Remark \ref{rem:Qc}, the operator $\mathcal{A}_s^\mu$ might have eigenvalues of infinite multiplicity when $L$ is rational. For instance in the case of $L=0.5$, the set of eigenvalues of infinite multiplicity is given by
\[
	\{\lambda = \omega^2,\;\omega=2(2n+1)\pi,\;n\in\N\}.
\]
We can predict that such an eigenvalue becomes a (small) spectral band in the 2D case for $\eps$ small enough. Indeed, as shown in \cite{Post:2006}, the dimension of the spectral projector on any interval is preserved for $\eps$ small enough. On the other hand, in most cases, a periodic 2D operator does not have eigenvalues (see for instance \cite{Sobolev:2002,Suslina:2002,friedlander:2003} for the proof of the absolute continuity of the spectrum of classes of periodic operator defined in waveguides, but see also the counterexample \cite{Filonov2001Absolute}). Thus the most likely possibility is that the operator $A^{\mu}_{\eps,s}$ has a small spectral band {(of width $O(\varepsilon)$, see Theorem~\ref{thmKuchment})} in a neighborhood of the eigenvalues of infinite multiplicity. This phenomenon can be seen on Figure~\ref{fig:infmult} where  the essential spectrum of $A^{\mu}_{\eps,s}$ is represented for different values of $\eps$ and for $L=0.5$. A small spectral band appears in the vicinity of $\omega = 2 \pi$, corresponding to the first eigenvalue of infinite multiplicity of the limit operator.
\begin{figure}[htbp]
\begin{center}
\includegraphics[scale=0.35,trim = 0cm 0cm 0cm 2cm, clip]{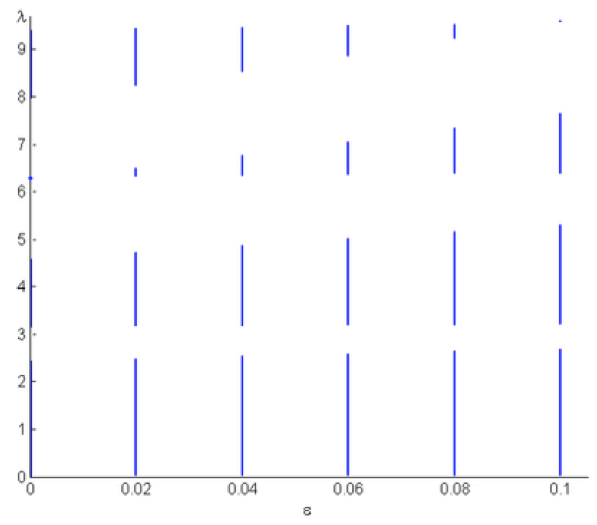}
\caption{\label{fig:infmult} Representation of the essential spectrum of $A_{\eps,s}^\mu$ for different values of $\eps$ for $L=0.5$}
\end{center}
\end{figure}

\subsubsection*{Discrete spectrum}
It is less easy to compute the discrete spectrum because one has to solve an eigenvalue problem set on an unbounded domain. To address this difficulty, we have used a method based on the construction of Dirichlet-to-Neumann operators in periodic waveguides (see \cite{ Fliss:2006,Fliss:2009}): { this requires the solution of cell problems (discretized here again using  the standard $P_1$ finite element methods) and the solution of a stationary Ricatti equation.} The construction of these Dirichlet-to-Neumann operators  enables us to reduce the numerical computation to a small neighborhood of the perturbation independently from the confinement of the mode (which depends on the distance between the eigenvalue and the essential spectrum of the operator). However the reduction of the problem leads to a non linear eigenvalue problem (since the DtN operators depend on the eigenvalue) of a fixed point nature. { It is solved using a Newton-type procedure, each iteration needing a finite element computation, }see \cite{SoniaNonLineraire} for more details. 
\begin{figure}[htbp]
\begin{center}
\includegraphics[scale=0.35,trim = 0cm 0cm 0cm 1.8cm, clip]{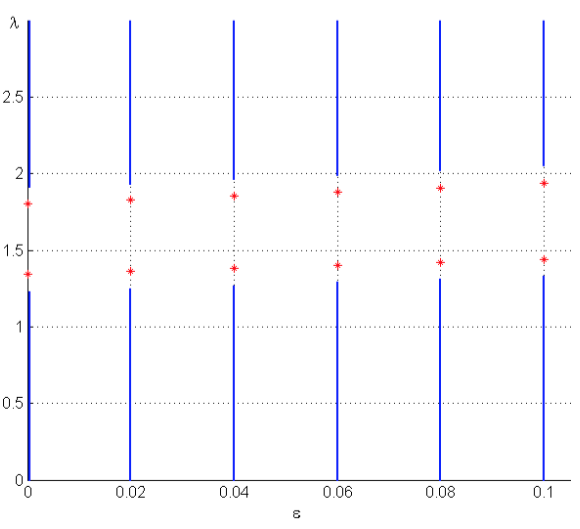}
\caption{\label{fig:eigenvalues1} Representation of the eigenvalues appearing in the first gap of the operator $A_{\eps,s}^\mu$ for different values of $\eps$ for $L=2$ and $\mu=0.25$ (red asterisks). The values for $\eps=0$ correspond to the limit operator $A_s^\mu$.}
\end{center}
\end{figure}
In Figure~\ref{fig:eigenvalues1}, we represent the eigenvalues computed for different values of $\eps$: here again, the blue bars correspond to the values $\omega$ such that $\lambda = \omega^2$ is in the essential spectrum of $A_{\eps,s}^\mu$. The red asterisks stand for the values $\omega$ such that $\lambda = \omega^2$ is in the discrete spectrum of $A_{\eps,s}^\mu$. In Figure~\ref{fig:eig_conv}, we make a zoom on the eigenvalues and we observe the linear convergence of one of this eigenvalue toward the limit one:  as explained in Remark \ref{rem:conv}, the error estimate~\eqref{ConvergenceValeurPropre} is suboptimal. Indeed, a high order asymptotic expansion of $\lambda^\varepsilon$ would restore the linear convergence rate. In Figure~\ref{fig:eigenvector}, the eigenfunction corresponding to the first eigenvalue of the operator $A_{\eps,s}^\mu$ is represented.
\begin{figure}[htbp]
\begin{center}
\includegraphics[scale=0.35,trim = 0cm 0cm 0cm 1.8cm, clip]{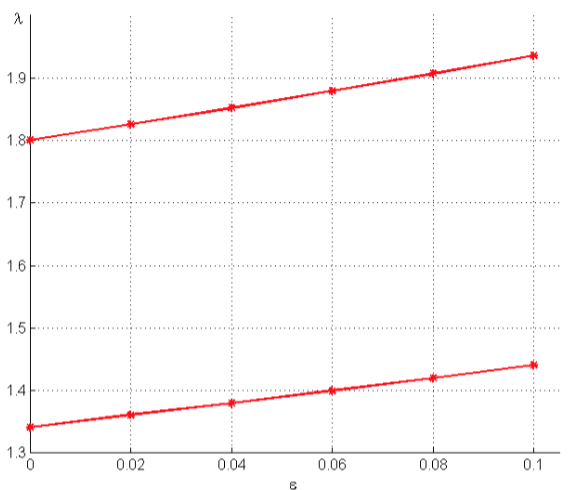}
\caption{\label{fig:eig_conv} Linear convergence of the eigenvalues represented in figure~\ref{fig:eigenvalues1} as $\eps\rightarrow 0$.}
\end{center}
\end{figure}
\begin{figure}[htbp]
\begin{center}
\includegraphics[scale=0.35]{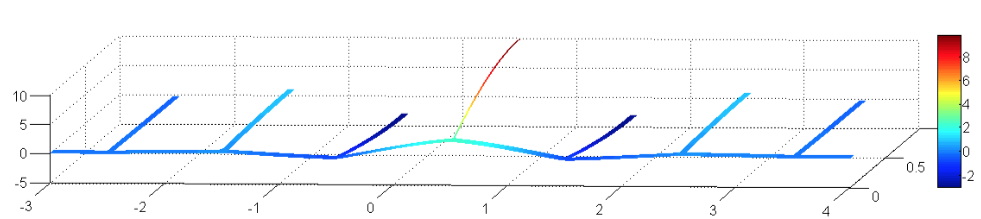}
\caption{\label{fig:eigenvector} Eigenfunction corresponding to the first eigenvalue of the operator $A_{\eps,s}^\mu$ for $L=2$, $\eps=0.06$ and $\mu=0.25$.}
\end{center}
\end{figure}
In Figure~\ref{fig:eigenvalues_mu}, we study the dependence of the eigenvalues with respect to $\mu\in(0,1)$. As it is natural to expect, the smaller $\mu$ is (so, the stronger the perturbation is), the better the eigenvalues are separated from the essential spectrum. When $\mu$ is close to $1$, the computation becomes more costly: since the distance between the eigenvalue and the essential spectrum becomes very small, the mesh size has to be small enough in order to make the distinction between the two different kinds of spectrum.  \begin{figure}[htbp]
\begin{center}
\includegraphics[scale=0.35,trim = 0cm 0cm 0cm 1.5cm, clip]{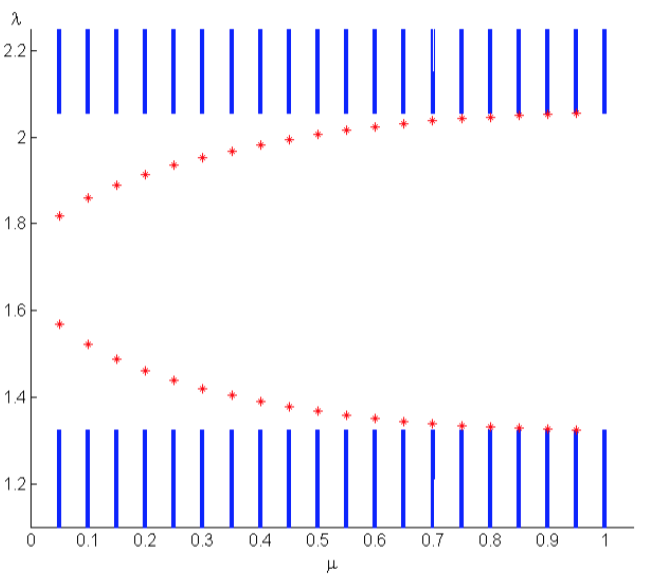}
\caption{\label{fig:eigenvalues_mu} Dependence of the eigenvalues in the first gap with respect to $\mu$ for $L=2$ and $\eps=0.1$.}
\end{center}
\end{figure}
~\\\\
A last natural question for which no theoretical answer has been given yet is what happens for larger values of $\eps$, i.e. when the spectrum of the operator $A_{\eps,s}^\mu$ is not close to the spectrum of the limit operator. In particular, if a gap exists for small values of $\eps$, does it still exist for large values of $\eps$ (until the obstacles disappear)? Similarly, do the eigenvalues still exist when $\eps$ increases or do they immerse into the essential spectrum?
In the cases that we have tested, the gaps seem to keep present for any value of $\eps$ for which the obstacles are present (for $\eps\in(0,\min(1,L/2))$). In Figure~\ref{fig:gap_eps}, we represent the dependence of the first two gaps with respect to $\eps$ in the case $L=2$. The limit case $\eps\rightarrow \min(1,L/2)$ has been studied by S. Nazarov \cite{MR2766589} where opening of a gap is proven when Dirichlet conditions are imposed on the boundary of the periodic waveguide instead of Neumann boundary conditions. The behaviour of the eigenvalues is by contrast more unclear. In Figure~\ref{fig:eig_eps}, we show the eigenvalues in the first gap of the operator $A_{\eps,s}^\mu$ for $L=2$ and $\mu=0.25$. We might think that the eigenvalues disappear for some values of $\eps<1$: however, as previously mentioned, the numerical computation becomes costly when the eigenvalues approach the essential spectrum. For this reason, it is difficult to make the distinction between the case when the eigenvalues do not exist any more and the case when they exist but are very close to the essential spectrum. Moreover, if they really disappear, do they immerse in the essential spectrum? Do they move in the complex plane? { Let us point out that the use of the sophisticated numerical method (based on an automatic choice of the mesh size) presented in \cite{KlindworthSchmidt} might help to answer these questions.}
\begin{figure}[htbp]
\begin{center}
\includegraphics[scale=0.35,trim = 0cm 0cm 0cm 0.9cm, clip]{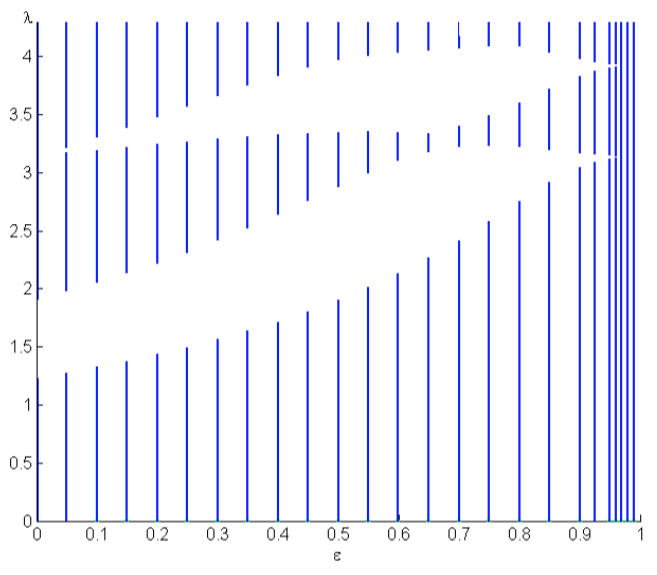}
\caption{\label{fig:gap_eps} Dependence of the first gaps with respect to $\eps<1$ for $L=2$.}
\end{center}
\end{figure}
\begin{figure}[htbp]
\begin{center}
\includegraphics[scale=0.35,trim = 0cm 0cm 0cm 0.9cm, clip]{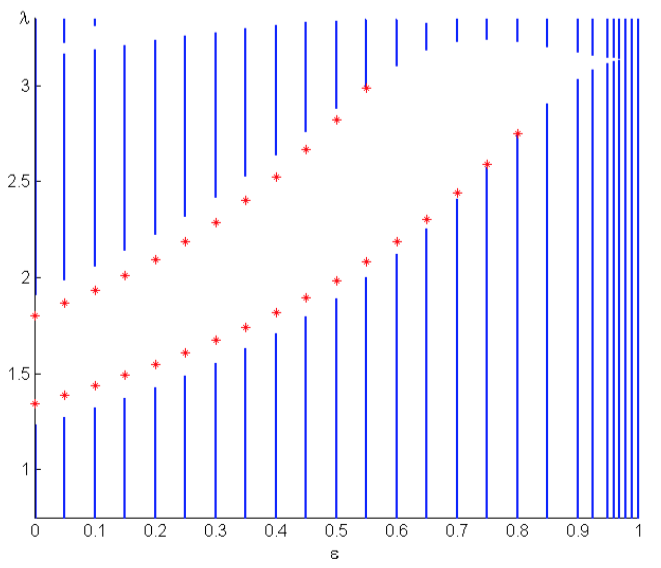}
\caption{\label{fig:eig_eps} Dependence of the eigenvalues in the first gap with respect to $\eps<1$ for $L=2$ and $\mu=0.25$.}
\end{center}
\end{figure}

\bibliographystyle{plain}
\bibliography{biblioFinal}

\def\cprime{$'$} \def\cprime{$'$} \def\cprime{$'$} \def\cprime{$'$}
  \def\cprime{$'$}
\begin{thebibliography}{10}

\bibitem{Ammari:2004}
H.~Ammari and F.~Santosa.
\newblock {Guided waves in a photonic bandgap structure with a line defect}.
\newblock {\em SIAM J. Appl. Math.}, 64(6):2018--2033 (electronic), 2004.

\bibitem{AvishaiLuck:1992}
Y.~Avishai and J.~M. Luck.
\newblock Quantum percolation and ballistic conductance on a lattice of wires.
\newblock {\em Physical Review B}, 45(3):1074, 1992.

\bibitem{NazarovBakharevRuotsalainen:2013}
F.~L. Bakharev, S.~A. Nazarov, and K.~M. Ruotsalainen.
\newblock A gap in the spectrum of the {N}eumann-{L}aplacian on a periodic
  waveguide.
\newblock {\em Appl. Anal.}, 92(9):1889--1915, 2013.

\bibitem{KuchmentBookGraph}
Gregory Berkolaiko and Peter Kuchment.
\newblock {\em Introduction to quantum graphs}, volume 186 of {\em Mathematical
  Surveys and Monographs}.
\newblock American Mathematical Society, Providence, RI, 2013.

\bibitem{panasenko2006}
Leonid Berlyand, Giuseppe Cardone, Yuliya Gorb, and Gregory Panasenko.
\newblock Asymptotic analysis of an array of closely spaced absolutely
  conductive inclusions.
\newblock {\em Netw. Heterog. Media}, 1(3):353--377, 2006.

\bibitem{BirmanSolomjakBook}
M.~Sh. Birman and M.~Z. Solomjak.
\newblock {\em Spectral theory of selfadjoint operators in {H}ilbert space}.
\newblock Mathematics and its Applications (Soviet Series). D. Reidel
  Publishing Co., Dordrecht, 1987.
\newblock Translated from the 1980 Russian original by S. Khrushch{\"e}v and V.
  Peller.

\bibitem{Borg:1946}
G.~Borg.
\newblock Eine {U}mkehrung der {S}turm-{L}iouvilleschen {E}igenwertaufgabe.
  {B}estimmung der {D}ifferentialgleichung durch die {E}igenwerte.
\newblock {\em Acta Math.}, 78:1--96, 1946.

\bibitem{BrownHoangPlumWood2014}
B.~M. Brown, V.~Hoang, M.~Plum, and I.~Wood.
\newblock Spectrum created by line defects in periodic structures.
\newblock {\em Math. Nachr.}, 287(17-18):1972--1985, 2014.

\bibitem{BrownHoangPlumWood2015}
B.~M. Brown, V.~Hoang, M.~Plum, and I.~Wood.
\newblock On the spectrum of waveguides in planar photonic bandgap structures.
\newblock {\em Proc. A.}, 471(2176):20140673, 20, 2015.

\bibitem{carlson:1998}
Robert Carlson.
\newblock Adjoint and self-adjoint differential operators on graphs.
\newblock {\em Electronic Journal of Differential Equations}, 6(1998):1--10,
  1998.

\bibitem{RR}
B.~Delourme, S.~Fliss, P.~Joly, and E.~Vasilevskaya.
\newblock Trapped modes in thin and infinite ladder like domains: existence and
  asymptotic analysis.
\newblock {\em INRIA Research Report}, 2016.

\bibitem{Eastham:1973}
M.~S.~P. Eastham.
\newblock {\em The spectral theory of periodic differential equations}.
\newblock Edinburgh : Scottish Academic Press, Edinburgh-London, distributed by
  chatto and windus edition, 1973.

\bibitem{Exner:1995}
Pavel Exner.
\newblock Lattice kronig-penney models.
\newblock {\em Physical review letters}, 74(18):3503, 1995.

\bibitem{Exner:1996}
Pavel Exner.
\newblock Contact interactions on graph superlattices.
\newblock {\em Journal of Physics A: Mathematical and General}, 29(1):87, 1996.

\bibitem{Figotin:1997}
A.~Figotin and A.~Klein.
\newblock Localized classical waves created by defects.
\newblock {\em J. Statist. Phys.}, 86(1-2):165--177, 1997.

\bibitem{Figotin:1998b}
A.~Figotin and A.~Klein.
\newblock Midgap defect modes in dielectric and acoustic media.
\newblock {\em SIAM J. Appl. Math.}, 58(6):1748--1773 (electronic), 1998.

\bibitem{Figotin:1996a}
A.~Figotin and P.~Kuchment.
\newblock Band-gap structure of spectra of periodic dielectric and acoustic
  media. {I}. {S}calar model.
\newblock {\em SIAM J. Appl. Math.}, 56(1):68--88, 1996.

\bibitem{Figotin:1996b}
A.~Figotin and P.~Kuchment.
\newblock Band-gap structure of spectra of periodic dielectric and acoustic
  media. {II}. {T}wo-dimensional photonic crystals.
\newblock {\em SIAM J. Appl. Math.}, 56(6):1561--1620, 1996.

\bibitem{Filonov2001Absolute}
N.~Filonov.
\newblock Second-order elliptic equation of divergence form having a compactly
  supported solution.
\newblock {\em J. Math. Sci. (New York)}, 106(3):3078--3086, 2001.
\newblock Function theory and phase transitions.

\bibitem{Fliss:2009}
S.~Fliss.
\newblock {\em Etude math\'{e}matique et num\'{e}rique de la propagation des
  ondes dans des milieux p\'{e}riodiques localement perturb\'{e}s}.
\newblock PhD thesis, Ecole Polytechnique, 5 2009.

\bibitem{SoniaNonLineraire}
Sonia Fliss.
\newblock A {D}irichlet-to-{N}eumann approach for the exact computation of
  guided modes in photonic crystal waveguides.
\newblock {\em SIAM J. Sci. Comput.}, 35(2):B438--B461, 2013.

\bibitem{friedlander:2003}
Leonid Friedlander.
\newblock Absolute continuity of the spectra of periodic waveguides.
\newblock {\em Contemporary Mathematics}, 339:37--42, 2003.

\bibitem{GiraultRaviart}
V.~Girault and P.A. Raviart.
\newblock {\em Finite element methods for {N}avier-{S}tokes equations},
  volume~5 of {\em Springer Series in Computational Mathematics}.
\newblock Springer-Verlag, Berlin, 1986.
\newblock Theory and algorithms.

\bibitem{Grieser2008}
Daniel Grieser.
\newblock Spectra of graph neighborhoods and scattering.
\newblock {\em Proc. Lond. Math. Soc. (3)}, 97(3):718--752, 2008.

\bibitem{HoangPlumWieners2009}
Vu~Hoang, Michael Plum, and Christian Wieners.
\newblock A computer-assisted proof for photonic band gaps.
\newblock {\em Z. Angew. Math. Phys.}, 60(6):1035--1052, 2009.

\bibitem{Joannopoulos:1995}
J.~D. Joannopoulos, R.~D. Meade, and J.~N. Winn.
\newblock {\em Photonic Crystal - Molding the Flow of Light}.
\newblock Princeton Univeristy Press, 1995.

\bibitem{Johnson:2002}
S.G. Johnson and J.~D. Joannopoulos.
\newblock {\em Photonic Crystal - The road from theory to practice}.
\newblock Kluwer Acad. Publ., 2002.

\bibitem{Fliss:2006}
P.~Joly, J.-R. Li, and S.~Fliss.
\newblock Exact boundary conditions for periodic waveguides containing a local
  perturbation.
\newblock {\em Commun. Comput. Phys.}, 1(6):945--973, 2006.

\bibitem{Khrabustovskyi2014}
Andrii Khrabustovskyi.
\newblock Opening up and control of spectral gaps of the {L}aplacian in
  periodic domains.
\newblock {\em J. Math. Phys.}, 55(12):121502, 23, 2014.

\bibitem{KhrabustovskyiKhruslov2015}
Andrii Khrabustovskyi and Evgeni Khruslov.
\newblock Gaps in the spectrum of the {N}eumann {L}aplacian generated by a
  system of periodically distributed traps.
\newblock {\em Math. Methods Appl. Sci.}, 38(1):11--26, 2015.

\bibitem{KlindworthSchmidt}
Dirk Klindworth and Kersten Schmidt.
\newblock An efficient calculation of photonic crystal band structures using
  {T}aylor expansions.
\newblock {\em Commun. Comput. Phys.}, 16(5):1355--1388, 2014.

\bibitem{Kuchment:1993}
P.~Kuchment.
\newblock {\em Floquet theory for partial differential equations}, volume~60 of
  {\em Operator Theory: Advances and Applications}.
\newblock Birkh\"auser Verlag, Basel, 1993.

\bibitem{Kuchment:2003}
P.~Kuchment and B.~Ong.
\newblock On guided waves in photonic crystal waveguides.
\newblock In {\em Waves in periodic and random media ({S}outh {H}adley, {MA},
  2002)}, volume 339 of {\em Contemp. Math.}, pages 105--115. Amer. Math. Soc.,
  Providence, RI, 2003.

\bibitem{KuchmentQuantumgraphs1}
Peter Kuchment.
\newblock Quantum graphs. {I}. {S}ome basic structures.
\newblock {\em Waves Random Media}, 14(1):S107--S128, 2004.
\newblock Special section on quantum graphs.

\bibitem{KuchmentQuantumgraphs2}
Peter Kuchment.
\newblock Quantum graphs. {II}. {S}ome spectral properties of quantum and
  combinatorial graphs.
\newblock {\em J. Phys. A}, 38(22):4887--4900, 2005.

\bibitem{KuchmentQuantumSurvey}
Peter Kuchment.
\newblock Quantum graphs: an introduction and a brief survey.
\newblock In {\em Analysis on graphs and its applications}, volume~77 of {\em
  Proc. Sympos. Pure Math.}, pages 291--312. Amer. Math. Soc., Providence, RI,
  2008.

\bibitem{KuchmentOng:2010}
Peter Kuchment and Beng-Seong Ong.
\newblock On guided electromagnetic waves in photonic crystal waveguides.
\newblock In {\em Operator theory and its applications}, volume 231 of {\em
  Amer. Math. Soc. Transl. Ser. 2}, pages 99--108. Amer. Math. Soc.,
  Providence, RI, 2010.

\bibitem{Kuchment:2001}
Peter Kuchment and Yehuda Pinchover.
\newblock Integral representations and {L}iouville theorems for solutions of
  periodic elliptic equations.
\newblock {\em J. Funct. Anal.}, 181(2):402--446, 2001.

\bibitem{KuchmentZeng:2001}
Peter Kuchment and Hongbiao Zeng.
\newblock Convergence of spectra of mesoscopic systems collapsing onto a graph.
\newblock {\em J. Math. Anal. Appl.}, 258(2):671--700, 2001.

\bibitem{NazarovLadder2012}
S~Nazarov.
\newblock On the spectrum of the laplace operator on the infinite dirichlet
  ladder.
\newblock {\em St. Petersburg Mathematical Journal}, 23(6):1023--1045, 2012.

\bibitem{NazarovPeriodic81}
S.~A. Nazarov.
\newblock Elliptic boundary value problems with periodic coefficients in a
  cylinder.
\newblock {\em Izv. Akad. Nauk SSSR Ser. Mat.}, 45(1):101--112, 239, 1981.

\bibitem{Nazarov:2010a}
S.~A Nazarov.
\newblock An example of multiple gaps in the spectrum of a periodic waveguide.
\newblock {\em Mat. Sb.}, 201(4):99--124, 2010.

\bibitem{MR2766589}
S.~A. Nazarov.
\newblock Opening of a gap in the continuous spectrum of a periodically
  perturbed waveguide.
\newblock {\em Mat. Zametki}, 87(5):764--786, 2010.

\bibitem{Nazarov:2012a}
S.~A. Nazarov.
\newblock The asymptotic analysis of gaps in the spectrum of a waveguide
  perturbed with a periodic family of small voids.
\newblock {\em J. Math. Sci. (N. Y.)}, 186(2):247--301, 2012.
\newblock Problems in mathematical analysis. No. 66.

\bibitem{NazarovLadder2014}
S.~A. Nazarov.
\newblock Bounded solutions in a {T}-shaped waveguide and the spectral
  properties of the {D}irichlet ladder.
\newblock {\em Comput. Math. Math. Phys.}, 54(8):1261--1279, 2014.

\bibitem{Nazarov2010Tshape}
SA~Nazarov.
\newblock Trapped modes in a t-shaped waveguide.
\newblock {\em Acoustical Physics}, 56(6):1004--1015, 2010.

\bibitem{nazarov1994elliptic}
Sergey Nazarov and Boris~A Plamenevsky.
\newblock {\em Elliptic problems in domains with piecewise smooth boundaries},
  volume~13.
\newblock Walter de Gruyter, 1994.

\bibitem{panasenko2007}
G.P. Panasenko and E.~Perez.
\newblock Asymptotic partial decomposition of domain for spectral problems in
  rod structures.
\newblock {\em Journal de Math\'ematiques Pures et Appliqu\'ees}, 87(1):1 --
  36, 2007.

\bibitem{Parnovski:2008}
L.~Parnovski.
\newblock Bethe-{S}ommerfeld conjecture.
\newblock {\em Ann. Henri Poincar\'e}, 9(3):457--508, 2008.

\bibitem{Parnovski:2010}
L.~Parnovski and A.~V. Sobolev.
\newblock {Bethe-Sommerfeld conjecture for periodic operators with strong
  perturbations}.
\newblock {\em Invent. Math.}, 181(3):467--540, 2010.

\bibitem{Post:2006}
Olaf Post.
\newblock Spectral convergence of quasi-one-dimensional spaces.
\newblock {\em Ann. Henri Poincar\'e}, 7(5):933--973, 2006.

\bibitem{PostBook}
Olaf Post.
\newblock {\em Spectral analysis on graph-like spaces}, volume 2039 of {\em
  Lecture Notes in Mathematics}.
\newblock Springer, Heidelberg, 2012.

\bibitem{Reed:1972}
M.~Reed and B.~Simon.
\newblock {\em Methods of modern mathematical physics v. I-IV}.
\newblock Academic Press, New York, 1972-1978.

\bibitem{SchatzmannRubinstein:2001}
Jacob Rubinstein and Michelle Schatzman.
\newblock Variational problems on multiply connected thin strips. {I}. {B}asic
  estimates and convergence of the {L}aplacian spectrum.
\newblock {\em Arch. Ration. Mech. Anal.}, 160(4):271--308, 2001.

\bibitem{Saito:2000}
Yoshimi Saito.
\newblock The limiting equation for {N}eumann {L}aplacians on shrinking
  domains.
\newblock {\em Electron. J. Differential Equations}, pages No. 31, 25 pp.
  (electronic), 2000.

\bibitem{Sobolev:2002}
Alexander~V. Sobolev and Jonathan Walthoe.
\newblock Absolute continuity in periodic waveguides.
\newblock {\em Proc. London Math. Soc. (3)}, 85(3):717--741, 2002.

\bibitem{Suslina:2002}
T.~A. Suslina and R.~G. Shterenberg.
\newblock Absolute continuity of the spectrum of the magnetic schr\"odinger
  operator with a~metric in a~two-dimensional periodic waveguide.
\newblock {\em Algebra i Analiz}, 14:159--206, 2002.

\bibitem{Vorobets:2011}
M.~Vorobets.
\newblock {On the Bethe-Sommerfeld conjecture for certain periodic Maxwell
  operators}.
\newblock {\em J. Math. Anal. Appl.}, 377(1):370--383, 2011.

\end{thebibliography}
\end{document}